\documentclass{amsart}

\usepackage{amsmath,amssymb}
\usepackage[all]{xy}
\usepackage{hyperref}

\hypersetup{
  pdftitle={Shifted dihedral isogeny quotients and the classification of exceptional rational functions of degree five},
  pdfauthor={Zhichao Tang and Xiang Fan},
  pdfkeywords={Exceptional rational functions, permutation rational functions, finite fields, dihedral monodromy, elliptic isogenies, Galois descent}
}

\numberwithin{equation}{section}
\newtheorem{thm}{Theorem}[section]
\newtheorem{prop}[thm]{Proposition}
\newtheorem{lem}[thm]{Lemma}
\newtheorem{cor}[thm]{Corollary}
\theoremstyle{remark}
\newtheorem{rem}[thm]{Remark}

\title[Shifted dihedral quotients and degree-five exceptional maps]{Shifted dihedral isogeny quotients and the classification of exceptional rational functions of degree five}

\author{Zhichao Tang}

\author{Xiang Fan}
\address{School of Mathematics, Sun Yat-sen University, Guangzhou 510275, China}


\begin{document}

\begin{abstract}
We give a complete classification, up to $k$-M\"obius equivalence, of
exceptional rational functions of degree five over every finite field. The
classification gives explicit normal forms in every characteristic, exact
parameter identifications, and the resulting class counts. The main
structural input is an arithmetic theory of shifted dihedral isogeny
quotients valid in every odd degree. For each odd $n\geqslant3$, the
separable degree-$n$ rational maps whose geometric monodromy group is
isomorphic to $D_n$ and whose nontrivial inertia groups are generated by
reflections are precisely the maps induced by cyclic $n$-isogenies on
shifted Kummer quotients. We determine the exact ambiguity of the isogeny
data under two-sided $k$-M\"obius equivalence by means of a signed
Frobenius-descent invariant. The theory allows arbitrary odd $n$, reduced
kernels when the characteristic divides $n$, and wild reflection inertia.
If Frobenius acts on the cyclic kernel by
$\lambda\in(\mathbb Z/n\mathbb Z)^\times$, the induced map is exceptional
exactly when both $\lambda-1$ and $\lambda+1$ are units modulo $n$.
\end{abstract}

\maketitle

\section{Introduction}
\label{sec:introduction}

Let \(k=\mathbb F_q\). For rational maps over \(k\), there is a useful
distinction between an isolated permutation and a permutation phenomenon
that persists under extension of the ground field. A nonconstant
\(f\in k(X)\) is a \emph{permutation rational function} if it permutes
\(\mathbf P^1(k)\), and is \emph{exceptional} if it permutes
\(\mathbf P^1(\mathbb F_{q^m})\) for infinitely many \(m\). The latter
condition is geometric: it is detected by the relation between the
arithmetic and geometric monodromy groups of the cover
\(f:\mathbf P^1\to\mathbf P^1\). It is also stable under independent
changes of coordinates on the source and target. Let
\(\operatorname{PGL}_2(k)\) denote the group of \emph{M\"obius transformations},
equivalently, the degree-one rational functions in
\(k(X)\), under composition. For nonconstant
\(f,g\in k(X)\), we write
\[
 f\sim_k g
 \quad\Longleftrightarrow\quad
 f=\mu\circ g\circ\eta
 \quad\text{for some }\mu,\eta\in\operatorname{PGL}_2(k),
\]
and call the resulting class a \emph{\(k\)-M\"obius class}.
An exceptional function over \(k\) automatically permutes
\(\mathbf P^1(k)\), since its restriction to the finite set
\(\mathbf P^1(k)\) is an injective self-map and hence bijective.

The exceptional classification is the stable part of the corresponding
finite-field permutation problem. Indeed, for each fixed degree, every
permutation rational function over a sufficiently large finite field is
exceptional \cite[Theorems~4 and~5]{Cohen1970Distribution}; see also
\cite[Lemma~1.6]{DingZieve2022LowDegree}. Degrees two, three, and four are
known: degree two is elementary
\cite[Lemma~1.2]{DingZieve2022LowDegree}; degree three was classified by
Ferraguti--Micheli \cite{FerragutiMicheli2020DegreeThreePRF}; see also
\cite[Theorem~1.3]{DingZieve2022LowDegree}. Degree four was determined
by Hou \cite{Hou2021DegreeFourPRF} and, independently, by Ding--Zieve
\cite[Theorem~1.4]{DingZieve2022LowDegree}. Sze
\cite{Sze2023DegreeFive} divided nonpolynomial degree-five maps into five
cases according to the factorization of their denominators. Cases~I and~II
were excluded over sufficiently large fields. In Case~III, he determined the
possible forms for $q\geqslant457$ in odd characteristic and for $q\geqslant16$
in characteristic $2$, and proved that the resulting forms are sufficient
over all finite fields for which their parameters are defined
\cite[Theorems~4.1 and~4.6]{Sze2023DegreeFive}. Case~IV was treated under
an additional hypothesis, while Case~V was left open there. Our degree-five
classification covers every finite field and every characteristic. It
determines all \(k\)-M\"obius classes, including exact equivalence criteria
for the parameters and the resulting class counts.

The obstruction behind the remaining cases is structural rather than
computational. Fix an algebraic closure $\bar k$ of $k$, let $x$ be transcendental over
$k$, and, for $f\in k(X)$, let $\overline\Omega$ be the \emph{Galois closure} of
$\bar k(x)/\bar k(f(x))$. The \emph{geometric monodromy group} of $f$ is
\(G_f:=\operatorname{Gal}(\overline\Omega/\bar k(f(x)))\).
For $f\in k(X)$ of degree $n\geqslant2$ and
$\beta\in\mathbf P^1(\bar k)$, put
\(f^{-1}(\beta):=\{\alpha\in\mathbf P^1(\bar k):f(\alpha)=\beta\}\),
the underlying set of the \emph{geometric fibre}. The point $\beta$ is a
\emph{branch point} if this set has fewer than $n$ elements; write
$\operatorname{Br}(f)$ for the \emph{branch locus}. It is \emph{totally
ramified} if its fibre is a singleton.

For a separable exceptional map of degree five, $G_f$ is isomorphic to
$C_5$ or $D_5$. Cyclic monodromy gives the classical monomial, R\'edei,
and characteristic-five additive forms. In the dihedral case, total
ramification forces a polynomial form. The remaining maps have geometric
monodromy group isomorphic to $D_5$ and only reflection inertia. Their
geometric Galois closures have genus one. If the branch locus contains a
\(k\)-rational point, the corresponding involution can be normalized to
negation on an elliptic curve. Without such a point there is no
\(k\)-rational normalization of this kind; the involution is instead a
\emph{shifted negation} \(P\mapsto\varepsilon-P\). Resolving this descent
obstruction is the main structural step of the paper.

We first state the resulting theory in arbitrary odd degree. Let
\(n\geqslant3\) be odd. An \emph{admissible isogeny datum of degree \(n\)}
over \(k\) is a triple \((E,N,\varepsilon)\), where \(E/k\) is an elliptic
curve, \(\varepsilon\in E(k)\), and \(N\subset E(\bar k)\) is a
Frobenius-stable cyclic subgroup of order \(n\), consisting of \(n\)
distinct geometric points. Let
\(\varphi:E\to E'\), where $E':=E/N$, be the \emph{quotient isogeny}, and put
\(\iota_\varepsilon(P):=\varepsilon-P\). The two quotients
$E/\langle\iota_\varepsilon\rangle$ and
$E'/\langle\iota_{\varphi(\varepsilon)}\rangle$ in the following diagram
are $k$-rational genus-zero curves, and the lower morphism $g$ has degree
\(n\):
\begin{equation}
\label{eq:intro-kummer-diagram}
\xymatrix{
 E\ar[r]^{\varphi}\ar[d]_{\pi}&E'=E/N\ar[d]^{\pi'}\\
 E/\langle\iota_\varepsilon\rangle\cong\mathbf P^1
 \ar[r]^{g}&
 E'/\langle\iota_{\varphi(\varepsilon)}\rangle\cong\mathbf P^1.
}
\end{equation}
Denote the resulting \(k\)-M\"obius class of $g$ by
\(\mathcal M(E,N,\varepsilon)\).

\begin{thm}[Shifted dihedral isogeny correspondence]
\label{thm:intro-dihedral}
Let \(k=\mathbb F_q\) and let \(n\geqslant3\) be odd.
\begin{enumerate}
\item[(1)] A separable degree-\(n\) rational map \(f\) has geometric monodromy
      group isomorphic to \(D_n\), and every nontrivial geometric inertia
      group is generated by a reflection, if and only if
      \(f\in\mathcal M(E,N,\varepsilon)\) for an admissible isogeny datum
      \((E,N,\varepsilon)\) of degree \(n\).
\item[(2)] The class \(\mathcal M(E,N,\varepsilon)\) depends on
      \(\varepsilon\) only through its image in \(E(k)/2E(k)\), and
      \[
       \operatorname{Br}(f)\cap\mathbf P^1(k)\ne\varnothing
       \quad\Longleftrightarrow\quad
       \varepsilon\in2E(k).
      \]
\item[(3)] There is a unique
      \(\lambda\in(\mathbb Z/n\mathbb Z)^\times\) such that
      \(\sigma_q(P)=[\lambda]P\) for every \(P\in N\). The map \(f\) is
      exceptional if and only if
      \[
       \gcd(n,\lambda-1)=\gcd(n,\lambda+1)=1.
      \]
\end{enumerate}
\end{thm}

The three assertions are proved in Section~\ref{sec:isogenies}. No
assumption that \(n\) is prime or prime to \(\operatorname{char}k\) is
made. In particular, no exceptional map in
Theorem~\ref{thm:intro-dihedral} exists when \(3\mid n\); for \(n=5\),
exceptionality is equivalent to \(\lambda=\pm2\).

Theorem~\ref{thm:shifted-trace-coordinate} in the same section makes
the shifted quotient explicit. If \(\varepsilon\notin N\), choose a
coordinate \(z\) on the source quotient whose pullback has simple poles
at \(O_E\) and \(\varepsilon\). The field trace from \(E\) to \(E/N\),
whose pullback after base change is \(\sum_{\alpha\in N}z\circ\tau_\alpha\),
is invariant under the target involution and descends to a coordinate
\(\Psi_{N,z}\) on the target quotient. Thus the lower map is written
directly in the coordinates \(z\) and \(\Psi_{N,z}\). The construction works
in every characteristic and involves no division by \(n\).

The correspondence is not injective on data; its exact ambiguity is a
\emph{signed descent}. For \(\varepsilon\in E(k)\), choose
\(\omega\in E(\bar k)\) with \(2\omega=\varepsilon\), and define the \emph{Frobenius defect} by
\[
 \mathfrak d_E(\varepsilon)
   :=[\omega^{\sigma_q}-\omega]
   \in\mathfrak D(E):=E[2]/(\sigma_q-1)E[2].
\]
This class is independent of the choice of \(\omega\).

\begin{thm}[Signed descent]
\label{thm:intro-signed-descent}
Let \(\mathcal D_i=(E_i,N_i,\varepsilon_i)\), \(i=1,2\), be admissible
isogeny data of the same odd degree. Then
\(\mathcal M(\mathcal D_1)=\mathcal M(\mathcal D_2)\) if and only if there
is an origin-preserving isomorphism
\(\alpha:E_{1,\bar k}\to E_{2,\bar k}\) such that
\begin{equation}
\label{eq:signed-data}
 \alpha(N_1)=N_2,\qquad
 {}^{\sigma_q}\alpha=\alpha\ \text{or}\ -\alpha,\qquad
 \alpha_*\mathfrak d_{E_1}(\varepsilon_1)
   =\mathfrak d_{E_2}(\varepsilon_2).
\end{equation}
Here
\({}^{\sigma_q}\alpha:=\sigma_q^{E_2}\circ\alpha\circ
(\sigma_q^{E_1})^{-1}\). In either sign, \(\alpha|_{E_1[2]}\) is
Galois equivariant and hence induces
\(\alpha_*:\mathfrak D(E_1)\to\mathfrak D(E_2)\).
\end{thm}

Several geometric precursors are relevant. Fried gave a modular
parametrization of tame prime-degree exceptional involution covers over
finite fields and described their two-sided M\"obius equivalence
\cite[Corollary~3.5]{Fried1994GlobalConstruction}. In characteristic zero,
Guralnick--M\"uller--Saxl proved the indecomposable four-reflection isogeny
description and already allowed a shifted involution, noting the associated
2-divisibility obstruction over the ground field
\cite[Theorem~6.5 and the subsequent remark]
{GuralnickMullerSaxl2003RationalFunctionSchur}. Pakovich obtained a
complex-geometric quotient description for rational functions with
genus-one normalization \cite{Pakovich2018Normalization}.
Bisson--Tibouchi constructed prime-degree unshifted isogeny quotients over
finite fields and characterized the permutation property by a kernel
condition; their later sections implement the construction with kernel
polynomials and the formulas of Kohel
\cite[Theorem~2.1 and \S\S2--4]{BissonTibouchi2018Isogenies}.

Thus neither the appearance of an elliptic-isogeny quotient nor a shifted
involution is new in isolation. Our contribution is an arithmetic theory of
shifted quotients over finite fields: an if-and-only-if
correspondence for arbitrary odd cyclic \(D_n\)-covers with reflection
inertia, valid in all characteristics and without an indecomposability
hypothesis, together with the exact signed-descent ambiguity under two-sided
\(k\)-M\"obius equivalence. The framework includes wild reflection inertia
and reduced kernels when the characteristic divides \(n\), and yields the
composite-degree Frobenius criterion in Theorem~\ref{thm:intro-dihedral}.
In degree five it resolves the cases with no $k$-rational branch point and
leads to complete normal forms, parameter identifications, and class counts.

We now give the degree-five consequence. The four blocks below are
distinguished intrinsically by separability, the isomorphism class of
the geometric monodromy group, total ramification, and rationality of
the branch locus; these invariants will also prove that the families
are disjoint.

\begin{thm}[Complete degree-five classification]
\label{main}
Let \(k=\mathbb F_q\) have characteristic \(p\), and let \(f\in k(X)\)
have degree \(5\). Then \(f\) is exceptional if and only if
\(f\sim_k g\), where \(g\) belongs to one of the following families.

\smallskip
\noindent\textbf{Cyclic monodromy.}
\begin{enumerate}
\item[\textup{(A)}] For \(q\not\equiv1\pmod5\), \(g=X^5\). This
      includes the inseparable case \(p=5\).
\item[\textup{(B)}] For \(q\not\equiv0,-1\pmod5\) and
      \(\delta\in\mathbb F_{q^2}\setminus\mathbb F_q\),
      \(g=R_{5,\delta}:=\nu_\delta^{-1}\circ X^5\circ\nu_\delta\), where
      \(\nu_\delta(X):=(X-\delta^q)/(X-\delta)\).
\item[\textup{(C)}] For \(p=5\) and
      \(a\in k^\times\setminus(k^\times)^4\), \(g=X^5-aX\).
\end{enumerate}

\noindent\textbf{Dihedral monodromy with total ramification.}
\begin{enumerate}
\item[\textup{(D)}] For \(p\ne5\), \(q\equiv\pm2\pmod5\), and
      \(a\in k^\times\),
      \[
       g=D_5(X,a):=X^5-5aX^3+5a^2X.
      \]
\item[\textup{(E)}] For \(p=5\) and nonsquare \(a\in k^\times\),
      \(g=X(X^2-a)^2\).
\end{enumerate}

\noindent\textbf{Dihedral reflection inertia with a rational branch point.}
\begin{enumerate}
\item[\textup{(F)}] For odd \(p\),
      \[
       g=\mathcal F_{r,a,d}(X):=
       X+\frac{aX^3+(3d/r)X^2+(3ar-32r^2)X+d}{(X^2-r)^2}.
      \]
      Here \(r\in k^\times\) is nonsquare, and \(a,d\in k\) satisfy
      \[
       a\ne52r,\qquad d^2=a^2r^3-16ar^4+128r^5.
      \]
\item[\textup{(G)}] For \(p=2\),
      \[
       g=X+\frac{W(X)}{(X^2+X+t)^2},
      \]
      where \(t\in k\) with \(\operatorname{Tr}_{k/\mathbb F_2}(t)=1\), and
      \[
       \text{either}\quad W=1,
       \qquad\text{or}\qquad
       W=t^{-1}X^3+X\quad\text{with }t\ne1.
      \]
\end{enumerate}

\noindent\textbf{Dihedral reflection inertia without a rational branch point.}
\begin{enumerate}
\item[\textup{(H)}] For \(p=2\),
      \[
       g=\mathcal H_{r,u}(X):=
       X+\frac{r(X^2+X+\eta)U_{r,u}(X)}
       {U_{r,u}(X)^2+rU_{r,u}(X)V_{r,u}(X)+r^3V_{r,u}(X)^2}.
      \]
      Here \(r,u\in k\) satisfy \(r\ne1\), \(u\ne0\),
      \(\operatorname{Tr}_{k/\mathbb F_2}(r)=1\), and
      \(\operatorname{Tr}_{k/\mathbb F_2}(r^5(r+1)/u^2)=1\), with
      \(\eta:=r^5(r+1)/u^2\). Moreover,
      \[
       U_{r,u}(X):=uX^2+u(1+\eta)+r^3,
       \qquad
       V_{r,u}(X):=X^2+X+r+\eta.
      \]
\item[\textup{(I)}] For odd \(p\), \(g=\mathcal I_{a,b,c,r,s}(X):=\)
      \[
       X+\frac{2sA(X)C(X)-r\bigl(A(X)D(X)+B(X)C(X)\bigr)+2B(X)D(X)}
       {B(X)^2-rA(X)B(X)+sA(X)^2}.
      \]
      Here \((a,b,c,r,s)\in k^5\) satisfies
      \begin{enumerate}
      \item[\textup{(I1)}]
      \(-16a^3c^2+16a^2b^2+72abc^2-64b^3-27c^4\ne0\);
      \item[\textup{(I2)}]
      \(X^2+rX+s\) is irreducible over \(k\);
      \item[\textup{(I3)}] \(-4as+2br-c^2-r^3+6rs=0\) and \(-ac^2+b^2-2bs+c^2r-r^2s+5s^2=0\);
      \item[\textup{(I4)}]
      \(X^4-2aX^2-4cX+a^2-4b\) has no root in \(k\).
      \end{enumerate}
      The auxiliary polynomials are
      \[
      \begin{aligned}
       A(X)&:=X^2+r-a, &\qquad B(X)&:=s-b-cX,\\
       C(X)&:=2(c+rX), & D(X)&:=c(X^2+a)+2(b+s)X.
      \end{aligned}
      \]
\end{enumerate}
\par\noindent
The nine families are pairwise disjoint up to \(k\)-M\"obius
equivalence.
\end{thm}

The coefficient forms in families \(\mathrm{(F)}\) and \(\mathrm{(G)}\) are precisely Sze's odd- and
even-characteristic Case~III forms \cite[Theorems~4.1 and~4.6]{Sze2023DegreeFive},
so the formulas themselves are not new. The restriction \(a\ne52r\) only
removes the totally ramified boundary, which lies in \(\mathrm{(D)}\) or
\(\mathrm{(E)}\); see Remark~\ref{rem:family-F-boundary}. For \(\mathrm{(F)}\)
and \(\mathrm{(G)}\), our additions are their uniform derivation within the
all-characteristic monodromy--isogeny classification, the exceptional
criterion, the exact \(k\)-M\"obius identifications, and the class counts.
By contrast, to the best of our knowledge, the explicit
no-rational-branch forms \(\mathrm{(H)}\) and \(\mathrm{(I)}\) do not appear
in the previous literature.

The explicit list still contains parameter redundancies, especially in the
shifted families. The next theorem determines the exact quotient by
\(k\)-M\"obius equivalence. For
family \(\mathrm{(G)}\), put
\[
 \mathcal U_t:=X+\frac1{(X^2+X+t)^2},
 \qquad
 \mathcal G_t:=X+\frac{t^{-1}X^3+X}{(X^2+X+t)^2}.
\]
For family \(\mathrm{(I)}\), let \(\Theta(a,b,c,r,s)\in k\) denote the
coefficient realization, defined by \eqref{eq:I-theta-invariant}, of the
geometric order-ten invariant introduced in
Section~\ref{sec:parameter-equivalence}.

\begin{thm}[\(k\)-M\"obius equivalence and class counts]
\label{thm:intro-equivalence-counts}
The complete \(k\)-M\"obius equivalence criteria and class counts for
families \(\mathrm{(A)}\)--\(\mathrm{(I)}\) are as follows.
\begin{enumerate}
\item[(1)] Families \(\mathrm{(A)}\) and \(\mathrm{(B)}\) each form one
      \(k\)-M\"obius class. In \(\mathrm{(C)}\), two parameters give
      \(k\)-M\"obius equivalent maps exactly when their quotient is a
      fourth power; in \(\mathrm{(D)}\) and \(\mathrm{(E)}\), exactly
      when their quotient is a square.
\item[(2)] In \(\mathrm{(F)}\),
      \(\mathcal F_{r,a,d}\sim_k\mathcal F_{r',a',d'}\) exactly when
      \(a/r=a'/r'\).
\item[(3)] All \(\mathcal U_t\) are \(k\)-M\"obius equivalent; one has
      \(\mathcal G_t\sim_k\mathcal G_{t'}\) exactly when \(t=t'\), and
      the two rows are disjoint.
\item[(4)] In \(\mathrm{(H)}\),
      \(\mathcal H_{r,u}\sim_k\mathcal H_{r',u'}\) exactly when \(r=r'\).
\item[(5)] In \(\mathrm{(I)}\), two odd-admissible quintuples give
      \(k\)-M\"obius equivalent maps exactly when their \(\Theta\)-invariants agree.
\item[(6)] Let \(N_{\mathrm A},\ldots,N_{\mathrm I}\) be the corresponding
numbers of \(k\)-M\"obius classes, with \(N_*:=0\) when the field condition
fails. If \(p=2\), put
\(\epsilon_2:=\operatorname{Tr}_{k/\mathbb F_2}(1)\), viewed as an integer
in $\{0,1\}$. If \(p\) is odd,
let \(\chi_2\) be the \emph{quadratic character} of \(k\), extended by
\(\chi_2(0):=0\), and put \(\epsilon_5:=1\) when \(\chi_2(5)\ne1\), and
\(\epsilon_5:=0\) otherwise. The class counts are given in the following table.
\begingroup
\setlength{\arraycolsep}{3pt}
\begin{equation}
\label{eq:intro-class-counts}
\begin{array}{c|c|c}
\text{family}&N_*&\text{field condition}\\ \hline
\mathrm{(A)}&1&q\not\equiv1\pmod5\\
\mathrm{(B)}&1&q\not\equiv0,-1\pmod5\\
\mathrm{(C)}&3&p=5\\
\mathrm{(D)}&\gcd(2,q-1)&q\equiv\pm2\pmod5\\
\mathrm{(E)}&1&p=5\\
\mathrm{(F)}&\dfrac{q+2+\chi_2(-1)}2-\epsilon_5&p\text{ odd}\\
\mathrm{(G)}&\dfrac q2+1-\epsilon_2&p=2\\
\mathrm{(H)}&\dfrac q2-\epsilon_2&p=2\\
\mathrm{(I)}&\dfrac{q+1}2&p=5\\
\mathrm{(I)}&\dfrac{q+\chi_2(-1)+2\chi_2(5)}2&p\ne2,5
\end{array}
\end{equation}
\endgroup
\end{enumerate}
\end{thm}

The \(k\)-M\"obius equivalence criteria and class counts are proved in
Section~\ref{sec:parameter-equivalence}; together with the Weil-bound
reduction they give the following consequence.

\begin{cor}[Degree five over large fields]
\label{cor:large-fields}
If \(q>361\), a degree-five rational function over \(\mathbb F_q\)
permutes \(\mathbf P^1(\mathbb F_q)\) if and only if it is
\(k\)-M\"obius equivalent to a member of one of the families
\(\mathrm{(A)}\)--\(\mathrm{(I)}\).
\end{cor}

\begin{proof}
If $q>361$, then $\sqrt q>19$, so
\cite[Lemma~1.6]{DingZieve2022LowDegree} implies that every degree-five
permutation rational function over $\mathbb F_q$ is exceptional. The result
now follows from Theorem~\ref{main}. Conversely, every listed map is
exceptional by that theorem and therefore permutes $\mathbf P^1(\mathbb F_q)$.
\end{proof}

For separable maps, the proof reduces the problem to \(C_5\) and \(D_5\).
For reflection inertia, Riemann--Hurwitz gives a genus-one Galois closure;
translations and the sheet actions \(j\mapsto j+1,-j,\lambda j\) then yield
the isogeny model and Frobenius criterion. Kernel-polynomial V\'elu formulas
and translation traces give the explicit families, while signed descent,
Kubert's normal form, and the \emph{Weil pairing} give the equivalence criteria and
class counts.

\section{Monodromy reduction and ramification}
\label{sec:monodromy}

We fix the monodromy notation and reduce the
degree-five problem to cyclic or dihedral geometric monodromy. Let $k$ be
a field with algebraic closure $\bar{k}$. For $m\geqslant1$, let $C_{m}$ be the
\emph{cyclic group} of order $m$. For $m\geqslant2$, put $D_{m}:=C_{m}\rtimes C_{2}$,
where the nonidentity element of $C_{2}$ acts on $C_{m}$ by inversion.
It is the \emph{dihedral group} of order $2m$; \(C_m\) is its
\emph{rotation subgroup}, and the elements of
$D_{m}\setminus C_{m}$ are its \emph{reflections}.

Let $f\in k(X)$ have degree $n>1$. It is \emph{separable} if the
extension $k(X)/k(f(X))$ is separable. In positive characteristic
$p$, this is equivalent to $f\notin k(X^{p})$; see \cite[Lemma~2.2]{DingZieve2022LowDegree}.
If $\deg f=5$ and $f$ is inseparable, then $\operatorname{char}k=5$
and $f=h(X^{5})$ with $h\in\operatorname{PGL}_{2}(k)$, so $f\sim_{k}X^{5}$.
When $k$ is finite of characteristic $5$, the map $X\mapsto X^{5}$
is a \emph{Frobenius automorphism} on every finite extension of $k$;
hence it is exceptional.

Assume from now on that $f$ is separable. Let $x$ be transcendental
over $k$, write $f=A/B$ with coprime $A,B\in k[X]$, and put
\(\Phi_f(Y):=A(Y)-f(x)B(Y)\in k(f(x))[Y]\). Then $\Phi_f$ is a
separable polynomial of degree $n$, and its roots are precisely the
conjugates of $x$ over $k(f(x))$. Fix an algebraic closure of
$k(f(x))$ containing $\bar k$, let $\Omega$ be the splitting field
of $\Phi_f$ over $k(f(x))$, and put $k_{\Omega}:=\Omega\cap\bar k$,
the \emph{constant field} of $\Omega$. The \emph{arithmetic monodromy
group} is $A_f:=\operatorname{Gal}(\Omega/k(f(x)))$, while the
\emph{geometric monodromy group} is
$G_f:=\operatorname{Gal}(\Omega/k_\Omega(f(x)))$.
The roots of $\Phi_f$ form a set $\mathcal S\subset\mathbf P^1(\Omega)$
of cardinality $n$.
The natural actions of $A_{f}$ and $G_{f}$ on $\mathcal{S}$ are
faithful and transitive, and thereby realize them as transitive subgroups
of $\operatorname{Sym}(\mathcal{S})\cong S_{n}$. One also has
$G_{f}\triangleleft A_{f}$
and $A_{f}/G_{f}\cong\operatorname{Gal}(k_{\Omega}/k)$; see \cite[Lemma~2.4]{DingZieve2022LowDegree}.
If $k=\mathbb{F}_{q}$, this quotient is cyclic.

Both $A_f$ and $G_f$ are invariant under \(k\)-M\"obius equivalence. The
arithmetic group depends on $k$. The present definition of $G_f$ agrees
with the definition
\(G_f=\operatorname{Gal}(\overline\Omega/\bar k(f(x)))\)
in Section~\ref{sec:introduction}. Indeed,
$\Omega=k(\mathcal S)$ and $k_\Omega=k(\mathcal S)\cap\bar k$. Since
$k_\Omega$ is algebraically closed in $\Omega$ and
$\Omega/k_\Omega$ is separably generated, $\Omega/k_\Omega$ is \emph{regular}.
Hence $\Omega$ and $\bar k(f(x))$ are linearly disjoint over
$k_\Omega(f(x))$, while
\(\bar k(\mathcal S)=\bar k\Omega=\overline\Omega\)
is the Galois closure of $\bar k(x)/\bar k(f(x))$. Extension and restriction
therefore identify
\[
 \operatorname{Gal}(\Omega/k_\Omega(f(x)))
 \cong
 \operatorname{Gal}(\overline\Omega/\bar k(f(x)))
\]
as the same permutation group on $\mathcal S$.

Let $A_{1}$ and $G_{1}$ be the \emph{point stabilizers} of the root
$x\in\mathcal{S}$ in $A_{f}$ and $G_{f}$. For $k=\mathbb F_q$, Cohen's
group-theoretic criterion
\cite[Lemma~6 and Theorems~4--5]{Cohen1970Distribution} yields the
following point-stabilizer form of exceptionality; see also
\cite[Lemma~2.4(2)]{DingZieve2022LowDegree}: $f$ is exceptional if and
only if $A_1$ and $G_1$ have exactly one common orbit on $\mathcal S$.
Here a \emph{common orbit} is a subset that is simultaneously an
$A_1$-orbit and a $G_1$-orbit. The orbit $\{x\}$ is always common. Thus,
for separable $f$ of degree greater than one, exceptionality implies
$A_f\ne G_f$.
\begin{prop}
\label{monodromy-reduction} If $f\in\mathbb{F}_{q}(X)$ is separable
and exceptional of degree $5$, then $G_{f}\cong C_{5}$ or $G_{f}\cong D_{5}$.
\end{prop}

\begin{proof}
Put $F_{20}=\operatorname{AGL}_1(\mathbb F_5)$. A transitive subgroup
$A\leq S_5$ contains a Sylow $5$-subgroup $P$. If $P$ is normal, then
$A\leq N_{S_5}(P)=F_{20}$, and transitivity gives
$A\cong C_5,D_5$, or $F_{20}$. Otherwise $A$ has six Sylow
$5$-subgroups, so $30\mid |A|$. The case $|A|=30$ is excluded by Sylow theory. If the number of
Sylow $3$-subgroups were $10$, they would contribute $20$ nonidentity
elements, while the six Sylow $5$-subgroups already contribute $24$,
which is impossible in a group of order $30$. If the Sylow $3$-subgroup
were normal, then $A$ would lie in its normalizer in $S_5$, which has
order $12$. Hence $A\cong A_5$ or $S_5$.

If $A_f\cong A_5$, simplicity forces the nontrivial transitive normal
subgroup $G_f$ to equal $A_f$, contrary to Cohen's criterion. If
$A_f\cong S_5$, then $G_f\cong A_5$; but the stabilizers $S_4$ and
$A_4$ have the same two orbits. Thus $A_f\cong C_5,D_5$, or $F_{20}$.
The case $A_f\cong C_5$ is impossible, since $C_5$ has no proper
nontrivial transitive normal subgroup. If $A_f\cong D_5$, then
$G_f\cong C_5$. If $A_f\cong F_{20}$, then
$G_f\cong C_5$ or $D_5$. This proves the assertion.
\end{proof}

Let $f\in k(X)$ be separable of degree $n\geqslant2$, and write
$\beta=f(\alpha)$, with
$\alpha,\beta\in\mathbf P^1(\bar k)$. A \emph{local parameter} at a
point of $\mathbf P^1$ is a rational function having a simple zero
there. If $u_\beta$ is a local parameter at $\beta$, the
\emph{ramification index} of $f$ at $\alpha$ is
$e_\alpha(f):=\operatorname{ord}_\alpha(u_\beta\circ f)$; it is
independent of $u_\beta$, and
$\sum_{\alpha\in f^{-1}(\beta)}e_\alpha(f)=n$
\cite[Theorem~3.1.11]{Stichtenoth2009AlgebraicFunctionFieldsCodes}.
The point $\alpha$ is \emph{ramified} if $e_\alpha(f)>1$; equivalently,
$\beta$ is a branch point if some point above it is ramified. Its
\emph{ramification partition} is
$\operatorname{Ram}_\beta(f):=\{\!\{e_\alpha(f):\alpha\in
f^{-1}(\beta)\}\!\}$.

Let $\bar\Omega$ be the Galois closure of
$\bar k(x)/\bar k(f(x))$, and let $\bar C$ be the smooth projective
curve with function field $\bar\Omega$. Its Galois group is $G_f$.
A \emph{place} of a one-variable function field is the discrete
valuation corresponding to a point of its smooth projective curve.
Let $P_\beta$ be the place of $\bar k(f(x))$ corresponding to $\beta$,
and choose a place $\widetilde P$ of $\bar\Omega$ above it. The
\emph{decomposition group} is the stabilizer of $\widetilde P$ in
$G_f$; its subgroup acting trivially on the residue field is the
\emph{inertia group} $I_\beta$. The group $G_f$ acts transitively on
the places above $P_\beta$, and their inertia groups are conjugate.
Thus $I_\beta$ is well defined up to conjugacy
\cite[Theorem~3.7.1 and Definition~3.8.1(b)]{Stichtenoth2009AlgebraicFunctionFieldsCodes}.
Since the constant field is algebraically closed, the residue fields
are $\bar k$, so decomposition and inertia coincide. Put
$H:=\operatorname{Gal}(\bar\Omega/\bar k(x))$. The points above $\beta$
correspond to the $I_\beta$-orbits on
$G_f/H$. For the orbit represented by $gH$, the corresponding point
$\alpha$ has
\[
 e_\alpha(f)
 =[I_\beta:I_\beta\cap gHg^{-1}]
 =|I_\beta\cdot gH|,
\]
by multiplicativity of ramification indices
\cite[Proposition~3.1.6(b) and Theorem~3.8.2(b)]{Stichtenoth2009AlgebraicFunctionFieldsCodes}.

Put $p:=\operatorname{char}k$. In positive characteristic, a ramified
point is \emph{tame} if $p\nmid e_\alpha(f)$ and \emph{wild}
otherwise; in characteristic zero every ramified point is tame. A
branch point is tame if every point above it is tame. The first
ramification group is a normal $p$-subgroup of inertia, and the quotient
is cyclic of order prime to $p$; in particular, tame inertia is cyclic
\cite[Proposition~3.8.5(d),(e)]{Stichtenoth2009AlgebraicFunctionFieldsCodes}.
The \emph{different exponent} $d_\alpha(f)$ is the valuation at
$\alpha$ of the different ideal of the corresponding extension of
discrete valuation rings; see
\cite[Definitions~3.1.5 and~3.4.3]{Stichtenoth2009AlgebraicFunctionFieldsCodes}.
The \emph{Riemann--Hurwitz formula} for $f:\mathbf P^1\to\mathbf P^1$ gives
\[
 2n-2=\sum_{\alpha\ \mathrm{ramified}}d_\alpha(f);
\]
see
\cite[Theorem~3.4.13 and Corollary~3.4.14]{Stichtenoth2009AlgebraicFunctionFieldsCodes}.
Moreover,
\[
 d_\alpha(f)=e_\alpha(f)-1\quad\text{in the tame case},
 \qquad
 d_\alpha(f)\geqslant e_\alpha(f)\quad\text{in the wild case},
\]
by \cite[Corollary~3.5.5(a)]{Stichtenoth2009AlgebraicFunctionFieldsCodes}.

The inertia groups normally generate $G_f$. Indeed, let $J$ be the
normal subgroup generated by their conjugates. Then
$\bar C/J\to\mathbf P^1_{\bar k}$ is everywhere unramified, and
Riemann--Hurwitz forces its degree to be one. Thus $J=G_f$.

\begin{prop}
\label{cyclic-ramification}
Let $f\in k(X)$ be separable of degree $n>1$ with $G_f\cong C_n$. If
$p\nmid n$, then $f$ has exactly two branch points, both totally ramified. If $p\mid n$, then $n=p$ and $f$ has exactly
one branch point, which is totally ramified.
\end{prop}

\begin{proof}
Since $|G_f|=n$ and the action on the $n$ sheets is transitive, the
point stabilizer is trivial. Thus the geometric degree-$n$ extension
itself is Galois, and a generator of $G_f$ acts on
$\mathbf P^1_{\bar k}$ as an element of
$\operatorname{PGL}_2(\bar k)$ of order $n$. If its two eigenlines are
distinct, the eigenvalue ratio has order $n$, so $p\nmid n$, and every
nonidentity power has the same two fixed points. These are the only ramified
points, both totally ramified, and their contribution exhausts
Riemann--Hurwitz. Otherwise the generator is nontrivial unipotent, hence has
order $p$; thus $n=p$, and every nonidentity element has the same unique fixed
point. It is therefore the unique ramified point and is totally ramified.
\end{proof}

A field is \emph{perfect} if every finite algebraic extension is separable.
Finite and algebraically closed fields are perfect.

Suppose now that $G_f\cong D_5$. Its transitive degree-five action is
unique up to conjugacy: point stabilizers are reflections, and all
reflections are conjugate. A reflection has cycle type $2^2 1$, whereas
$C_5$ and $D_5$ are transitive. Since these are the nontrivial subgroup
types of $D_5$, the inertia-orbit description gives
\[
 I_\beta\cong C_2
 \Longrightarrow
 \operatorname{Ram}_\beta(f)=\{\!\{2,2,1\}\!\},
 \qquad
 I_\beta\cong C_5\ \text{or}\ D_5
 \Longrightarrow
 \operatorname{Ram}_\beta(f)=\{\!\{5\}\!\}.
\]
In the first case we call $\beta$ a \emph{reflection branch point}.

\begin{prop}
\label{dihedral-ramification}
Let $k$ be perfect of characteristic $p$. For a separable
$f\in k(X)$ of degree $5$ with $G_f\cong D_5$, let $n_t$ be the
number of its totally ramified branch points and $n_r$ the number of
its reflection branch points. Then
$|\operatorname{Br}(f)|=n_t+n_r$, with $n_t\leqslant1$, and:
\begin{enumerate}
\item[(1)] if $p\ne2,5$, then $(n_t,n_r)=(1,2)$ or $(0,4)$;
\item[(2)] if $p=2$, then $(n_t,n_r)=(1,1),(0,1)$, or $(0,2)$;
\item[(3)] if $p=5$, then $(n_t,n_r)=(1,1)$ or $(0,4)$.
\end{enumerate}
\end{prop}

\begin{proof}
If $p\ne2,5$, the cover is tame. Thus inertia at a totally ramified
point is cyclic, hence $C_5$ rather than $D_5$; such a point contributes
$4$, and a reflection point contributes $2$. Therefore $4n_t+2n_r=8$. The solution $(2,0)$ is impossible, since all order-five inertia groups
lie in the rotation subgroup, whereas the inertia groups normally
generate $D_5$. This proves (1).

Assume $p=2$. At a totally ramified point, inertia cannot be $D_5$.
Indeed, its first ramification group would be a normal $2$-subgroup of
$D_5$, hence trivial; the quotient by it would then have to be cyclic
of odd order, a contradiction. Thus inertia is the tame group $C_5$
and contributes $4$. At a reflection branch point, the two ramified
points are wild of index $2$; each different exponent is at least $2$,
so the contribution is at least $4$. Hence
$4n_t+4n_r\leqslant8$. Moreover $n_r\ne0$, since otherwise all inertia groups would lie in
$C_5$. This gives (2).

Assume $p=5$. A reflection branch point is tame and contributes $2$,
whereas a totally ramified point is wild and contributes at least $5$.
Thus $5n_t+2n_r\leqslant8$, so $n_t\leqslant1$. If $n_t=0$, Riemann--Hurwitz gives $2n_r=8$.
It remains to exclude $(n_t,n_r)=(1,0)$. In that case the unique branch
point $\beta$ is $k$-rational. Its fibre is a singleton $\{\alpha\}$;
since $k$ is perfect, Galois invariance gives $\alpha\in\mathbf P^1(k)$.
After sending $\alpha$ and $\beta$ to infinity, $f$ becomes a polynomial
$h\in k[X]$ with no finite branch point. For a separable polynomial,
the finite ramification points are precisely the zeros of its
derivative. Hence $h'$ has no root in $\bar k$ and is a nonzero
constant. Consequently, \(h=a_5X^5+a_1X+a_0\) with
\(a_5a_1\ne0\). The additive polynomial $a_5X^5+a_1X$ has five distinct roots, and
translation by each root fixes $h$. These five automorphisms exhaust
the degree-five extension $\bar k(X)/\bar k(h(X))$, which is therefore
Galois with group $C_5$, contrary to $G_f\cong D_5$. Thus $n_t=1$
forces $n_r=1$, proving (3).
\end{proof}

\begin{cor}
\label{total-branch-polynomial}
Let $k$ be perfect. If $f\in k(X)$ is separable of degree $5$, with
$G_f\cong D_5$, and has a totally ramified branch point, then this point
is unique, belongs to $\mathbf P^1(k)$, and $f$ is \(k\)-M\"obius equivalent to a
polynomial.
\end{cor}

\begin{proof}
Uniqueness follows from Proposition~\ref{dihedral-ramification}. The
singleton set of totally ramified branch points is Galois-stable, hence
its element $\beta$ lies in $\mathbf P^1(k)$. Its unique preimage is
likewise $k$-rational because $k$ is perfect. Source and target M\"obius
transformations sending these two points to infinity turn $f$ into a
polynomial.
\end{proof}

\section{Shifted dihedral isogeny quotients and signed descent}
\label{sec:isogenies}

Retain the notation of Section~\ref{sec:introduction}; throughout this
section $k=\mathbb F_q$. For $m\geqslant1$, put
$E[m]:=\{P\in E(\bar k):[m]P=O_E\}$. For $a\in E(\bar k)$, define
\(\tau_a(P):=P+a\), and for a finite subgroup \(N\subset E(\bar k)\) put
\(T_N:=\{\tau_a:a\in N\}\). Frobenius acts coordinatewise on
\(E(\bar k)\); its induced \emph{semilinear action} on \(\bar k(E)\) is denoted
by \(\widetilde\sigma_q\). A curve $C/k$ is \emph{$k$-rational} if $k(C)=k(z)$ for some transcendental $z$; such a $z$ is a \emph{$k$-coordinate}. In particular, a genus-zero curve with a $k$-point is $k$-rational
\cite[Proposition~1.6.3]{Stichtenoth2009AlgebraicFunctionFieldsCodes}.

We use the standard convention that an \emph{isogeny} is a nonconstant
morphism of elliptic curves that preserves the origin. Such a morphism is a
group homomorphism by Theorem~III.4.8 of
\cite{Silverman2009ArithmeticEllipticCurves}. For an
admissible datum \((E,N,\varepsilon)\), the Frobenius-stable reduced
subgroup \(N\) defines a separable quotient isogeny \(E\to E/N\) over
\(k\) by Proposition~III.4.12 and Remark~III.4.13.2 of the same reference.
If \(p=\operatorname{char}k\) divides \(|N|\), then \(E\) is
\emph{ordinary} by Corollary~III.6.4(c) there.

For an admissible datum $(E,N,\varepsilon)$, denote by
$\mathcal{M}(E,N,\varepsilon)$ the $k$-M\"obius class determined by the
lower morphism in Theorem~\ref{thm:intro-dihedral}\textup{(1)}. When $k$-coordinates
are fixed on the two quotient lines, write $f_{E,N,\varepsilon}$ for
the resulting representative. The data need not parametrize
$k$-M\"obius classes injectively.

\begin{thm}[Shifted trace-coordinate theorem]
\label{thm:shifted-trace-coordinate}
Let $(E,N,\varepsilon)$ be an admissible isogeny datum of odd degree
$n\geqslant3$ over $k=\mathbb F_q$, let $\varphi:E\to E/N$ be the
quotient isogeny, and assume $\varepsilon\notin N$. Choose a
$k$-coordinate on $E/\langle\iota_\varepsilon\rangle$ with its only
pole at the common image of $O_E$ and $\varepsilon$, and denote its
pullback to $E$ by $z$. Then $(z)_\infty=(O_E)+(\varepsilon)$. Put
\begin{equation}
 \overline\Psi_{N,z}:=\operatorname{Tr}_{k(E)/k(E/N)}(z)\in k(E/N).
\label{eq:shifted-translation-trace}
\end{equation}
After base change to $\bar k$, its pullback to $E$ is
$\displaystyle \varphi^*\overline\Psi_{N,z}=\sum_{\alpha\in N}z\circ\tau_\alpha$.
The function $\overline\Psi_{N,z}$ is invariant under
$\iota_{\varphi(\varepsilon)}$ and hence descends to a function
$\Psi_{N,z}$ on $(E/N)/\langle\iota_{\varphi(\varepsilon)}\rangle$. Moreover,
\[
 (\overline\Psi_{N,z})_\infty
 =(O_{E/N})+(\varphi(\varepsilon)),
\]
and the pullback of this divisor to $E$ is
\[
 \sum_{\alpha\in N}\bigl(({-\alpha})+(\varepsilon-\alpha)\bigr).
\]
Thus $\Psi_{N,z}$ is a $k$-coordinate on the target quotient. Consequently,
in the coordinates $z$ and $\Psi_{N,z}$, the lower morphism is a degree-$n$
representative of $\mathcal M(E,N,\varepsilon)$.
\end{thm}

\begin{proof}
The poles $O_E$ and $\varepsilon$ are distinct, and their $N$-orbits are
disjoint because $\varepsilon\notin N$. Since the quotient isogeny
$\varphi:E\to E/N$ is separable, it is unramified. After scalar extension to $\bar k$, the $\bar k(E/N)$-embeddings of
$\bar k(E)$ are the pullbacks by $\tau_\alpha$, $\alpha\in N$; hence
$\varphi^*\overline\Psi_{N,z}=\sum_{\alpha\in N}z\circ\tau_\alpha$.
Frobenius permutes $N$, and translation by $N$ permutes the summands.
Moreover $(\varphi^*\overline\Psi_{N,z})\circ\iota_\varepsilon
=\varphi^*\overline\Psi_{N,z}$, since $\alpha\mapsto-\alpha$ permutes
$N$. As $\varphi\iota_\varepsilon=\iota_{\varphi(\varepsilon)}\varphi$
and $\varphi^*$ is injective, $\overline\Psi_{N,z}$ is invariant under
$\iota_{\varphi(\varepsilon)}$ and descends to $\Psi_{N,z}$ on the target
quotient.

The summand $z\circ\tau_\alpha$ has simple poles at $-\alpha$ and
$\varepsilon-\alpha$. Within each set the points are distinct because the elements of $N$ are distinct;
an equality between the two sets would put $\varepsilon$ in $N$. Hence all $2n$
points are distinct and no cancellation occurs. They form the orbit of $O_E$ under
$\Gamma=T_N\rtimes\langle\iota_\varepsilon\rangle$, and its stabilizer is
trivial. Their images on $E/N$ are $O_{E/N}$ and
$\varphi(\varepsilon)$, so
$(\overline\Psi_{N,z})_\infty=(O_{E/N})+(\varphi(\varepsilon))$.
The involution $\iota_{\varphi(\varepsilon)}$ exchanges
$O_{E/N}$ and $\varphi(\varepsilon)$, so neither point is fixed and
the quotient map is unramified there. The two simple poles therefore
descend to a single simple pole of $\Psi_{N,z}$ on the genus-zero
target. Hence $\Psi_{N,z}$ is a coordinate. The last assertion follows
from the quotient diagram.
\end{proof}

If $N=\langle\xi\rangle$ and $n=2m+1$, then the pullback to $E$ of the
target coordinate is
$\varphi^*\overline\Psi_{N,z}=z+\sum_{j=1}^{m}
(z\circ\tau_{j\xi}+z\circ\tau_{-j\xi})$; we use this form below.

\begin{lem}[Descent from the arithmetic point stabilizer]
\label{lem:arithmetic-stabilizer-descent}
Let $k=\mathbb F_q$, let $f\in k(X)$ be separable, and let
$\bar\Omega/\bar k(f(x))$ be the splitting field of $\Phi_f$. Put
\[
 G:=\operatorname{Gal}(\bar\Omega/\bar k(f(x))),\qquad
 \widetilde G:=\operatorname{Aut}_{k(f(x))}(\bar\Omega),
\]
and let $H:=\operatorname{Gal}(\bar\Omega/\bar k(x))$ have order $2$.
There is a closed subgroup $B\subset\widetilde G$, centralizing $H$,
such that restriction induces an isomorphism
$B\xrightarrow{\sim}\operatorname{Gal}(\bar k/k)$. For
$K=\bar\Omega^B$,
\[
 K\cap\bar k=k,\qquad K\bar k=\bar\Omega,\qquad K^H=k(x).
\]
Moreover $K/k(f(x))$ is finite, regular, and separable. The group $H$
acts by $k$-automorphisms on its smooth projective curve.
\end{lem}

\begin{proof}
For brevity put $t=f(x)$, and let $\widetilde H$ be the stabilizer of
$x$. Since coefficient conjugation stabilizes $\bar\Omega$, restriction
to constants gives exact sequences
$1\to G\to\widetilde G\to\Gamma_k\to1$ and
$1\to H\to\widetilde H\to\Gamma_k\to1$: for the latter, compose any
lift with an element of $G$ carrying its image of $x$ back to $x$.
Since $\operatorname{Aut}(H)=1$, this second extension is central.
Choose $\widetilde F\in\widetilde H$ mapping to the arithmetic Frobenius
$F\in\Gamma_k$. The homomorphism $\mathbb Z\to\widetilde H$,
$m\mapsto\widetilde F^m$, is continuous for the profinite topology:
for every open normal subgroup $U$ of $\widetilde H$, the inverse image
of $U$ contains $r\mathbb Z$, where $r$ is the order of $\widetilde F$
in $\widetilde H/U$. It therefore extends uniquely to a continuous
homomorphism $s:\widehat{\mathbb Z}\to
\overline{\langle\widetilde F\rangle}$. Under $\Gamma_k\simeq\widehat{\mathbb Z}$ with $1\mapsto F$, restriction composed with $s$ is the identity. Thus $B=s(\Gamma_k)$ maps isomorphically to $\Gamma_k$, $\widetilde G=G\rtimes B$, and $B\cap G=B\cap H=1$. Since $H$ is central in $\widetilde H$, also $\widetilde H=H\times B$. Hence $[\widetilde G:B]=|G|$, so $K/k(t)$ is finite.

A constant fixed by $B$ lies in $k$, so $K\cap\bar k=k$. The subgroup
fixing the compositum $K\bar k$ is $B\cap G=1$, whence
$K\bar k=\bar\Omega$. Finally,
$K^H=\bar\Omega^{\langle B,H\rangle}
=\bar\Omega^{\widetilde H}=k(x)$. The extension $K/k(t)$ is separable as
a subextension of $\bar\Omega/k(t)$ and is regular because
$K\cap\bar k=k$.
\end{proof}

\begin{lem}[Different in the dihedral intermediate cover]
\label{lem:dihedral-intermediate-different}
Let $Z\to Y\simeq\mathbf P^1$ be Galois with group $D_n$, $n$ odd,
and let $H$ be a reflection. Suppose that every nontrivial inertia group
is generated by a reflection and $Z/H\simeq\mathbf P^1$. If
$d_\beta$ is the different exponent above a branch point $\beta$, then
its contribution to the different of $Z/H\to Y$ is
\(\displaystyle\frac{n-1}{2}d_\beta\). Hence
\(\displaystyle\sum_\beta d_\beta=4\) and $g(Z)=1$.
\end{lem}

\begin{proof}
The inertia group has one fixed orbit and $(n-1)/2$ two-element orbits
on $D_n/H$. In the local tower, transitivity of the different gives
$d(Q/\beta)=d(Q/P)+e(Q/P)d(P/\beta)$
\cite[Corollary~3.4.12(b)]{Stichtenoth2009AlgebraicFunctionFieldsCodes}.
At a two-element orbit the upper map is unramified, so the lower
different is $d_\beta$; at the fixed orbit the upper different is
$d_\beta$ and the lower map is unramified. This proves the first
claim without tameness. Riemann--Hurwitz gives
$2n-2=\frac{n-1}{2}\sum_\beta d_\beta$ and then
$2g(Z)-2=-4n+n\sum_\beta d_\beta=0$.
\end{proof}

The following proof adapts the argument of
\cite[Corollary~3.5]{Fried1994GlobalConstruction}; compare also
\cite[Theorem~6.5]{GuralnickMullerSaxl2003RationalFunctionSchur}.
\begin{proof}[Proof of Theorem~\ref{thm:intro-dihedral} (1).]
Put $t=f(x)$, let $\bar\Omega/\bar k(t)$ be the splitting field of
$\Phi_f$, and set
$G=\operatorname{Gal}(\bar\Omega/\bar k(t))\cong D_n$ and
$H=\operatorname{Gal}(\bar\Omega/\bar k(x))=\langle s\rangle$.
Thus $s$ is a reflection. Lemma~\ref{lem:arithmetic-stabilizer-descent} yields a closed section $B$; let $C/k$ be the
smooth projective curve with function field $K=\bar\Omega^B$. Then
$C_{\bar k}$ has function field $\bar\Omega$, while
$K^H=k(x)$ identifies $C/H$ with $\mathbf P^1_x$. The inclusion
$k(t)\subset k(x)$ is the original $k$-cover. After base change to
$\bar k$, its target is the geometric quotient $C_{\bar k}/G$, since
$\bar\Omega^G=\bar k(t)$; no individual element of $G$ is asserted to
descend to $k$.

Lemma~\ref{lem:dihedral-intermediate-different} applies to
$C_{\bar k}\to\mathbf P^1_{t,\bar k}$ and gives $g(C)=1$, including
in characteristic $2$. The Hasse--Weil bound
\cite[Theorem~5.2.3]{Stichtenoth2009AlgebraicFunctionFieldsCodes}
gives
$|C(k)|\geq q+1-2\sqrt q=(\sqrt q-1)^2>0$. Choose
$O_E\in C(k)$ and put $E=(C,O_E)$.

Let $R\cong C_n$ be the rotation subgroup. It is characteristic in $G$
and $B$-stable. No nontrivial rotation fixes a geometric point. Indeed,
over $\bar k$ all residue fields are $\bar k$, so geometric point
stabilizers in the top Galois cover are inertia groups; by hypothesis,
every nontrivial inertia group is generated by a reflection and therefore
contains no nontrivial rotation. For $\rho\in R$, put $a=\rho(O_E)$ and
$u=\tau_{-a}\rho\in\operatorname{Aut}(E_{\bar k},O_E)$, so that
$\rho(P)=u(P)+a$. If $u\ne[1]$, then $[1]-u$ is a nonzero endomorphism,
hence an isogeny and therefore surjective
\cite[Theorem~II.2.3]{Silverman2009ArithmeticEllipticCurves}; this produces
a fixed point of $\rho$, a contradiction. Hence $R=T_N$ for a cyclic
subgroup $N\subset E(\bar k)$ of order $n$. Since semilinear conjugation
sends $\tau_P$ to $\tau_{\sigma(P)}$, the $B$-stability of $R$ is exactly
the Galois stability of $N$. Thus the quotient
$\varphi:E\to E'=E/N$ is a separable $k$-isogeny
\cite[Proposition~III.4.12 and Remark~III.4.13.2]
{Silverman2009ArithmeticEllipticCurves}.

Since $H$ centralizes $B$, $s$ descends to a $k$-automorphism of $E$.
It has a fixed point. Indeed, the fixed-field quotient $E\to E/H$ is a
separable Galois cover of degree $2$, also in characteristic $2$; without
a fixed point it would be unramified, contradicting Riemann--Hurwitz
because $E/H\cong\mathbf P^1$. Put
$\varepsilon=s(O_E)\in E(k)$ and write $s(P)=u(P)+\varepsilon$. The
identity $s^2=[1]$ gives $u^2=[1]$ and
$u(\varepsilon)+\varepsilon=O_E$. The case $u=[1]$ is impossible: since $s\ne[1]$, it would make $s$ a
nontrivial translation, which has no fixed point. As $u-[1]$ is
surjective and $(u+[1])(u-[1])=0$, one has $u=[-1]$; thus
$s=\iota_\varepsilon$. This remains valid in characteristic $2$, since
$[-1]\ne[1]$ on an elliptic curve
\cite[Proposition~III.4.2(a)]{Silverman2009ArithmeticEllipticCurves}.
Moreover $s\tau_as^{-1}=\tau_{-a}$, so $s$ descends to
$\iota_{\varphi(\varepsilon)}$ on $E'$. Let $F\subset k(x)$ be the target
field of the lower morphism. Since $F\bar k=\bar k(t)$, one has $F\subset k(x)\cap\bar k(t)$.
If $h=R_0(t)$ lies in this intersection, with $R_0\in\bar k(T)$, Galois
invariance of $h$ and $t$ gives $R_0\in k(T)$. Thus
$k(x)\cap\bar k(t)=k(t)$, so $F\subset k(t)$; equality follows from
$[k(x):F]=[k(x):k(t)]=n$. Hence the lower morphism is \(k\)-M\"obius
equivalent to $f$.

Conversely, let $(E,N,\varepsilon)$ be admissible of odd degree
$n\geqslant3$, put $E'=E/N$, and set $H=\langle\iota_\varepsilon\rangle$ and
$\Gamma=T_N\rtimes H\cong D_n$. The surjectivity of $[2]$ gives $2\omega=\varepsilon$ for some
$\omega\in E(\bar k)$. Since
$\tau_{-\omega}\iota_\varepsilon\tau_\omega=[-1]$, both shifted
Kummer quotients have genus zero; their quotient maps are defined over
$k$ and the images of the origins are rational, so both are
$k$-isomorphic to $\mathbf P^1$.

Each involution acts nontrivially on the corresponding function field
$L$. Hence $L/L^{\langle\iota\rangle}$ is Galois of degree $2$ and
therefore separable, also in characteristic $2$. Thus both vertical maps
are separable of degree $2$. The relation
$\varphi\iota_\varepsilon=\iota_{\varphi(\varepsilon)}\varphi$ yields
a unique lower morphism $g$ of degree $n$. Since separability passes to
intermediate function-field extensions and $\pi'\varphi$ is separable,
so is $g$.

Over $\bar k$, quotient in stages identifies the target with
$E_{\bar k}/\Gamma$. For $0\ne a\in N$,
$\tau_a\iota_\varepsilon\tau_a^{-1}=\tau_{2a}\iota_\varepsilon
\ne\iota_\varepsilon$, so $H$ is \emph{core-free}. Thus the geometric Galois
closure is $E_{\bar k}\to E_{\bar k}/\Gamma$ and $G_g\cong D_n$.
A geometric point stabilizer meets $T_N$ trivially and contains at most
one reflection, since the product of two distinct reflections is a
nontrivial translation. Hence every nontrivial inertia group is
generated by a reflection.
\end{proof}

\begin{proof}[Proof of Theorem~\ref{thm:intro-dihedral} (2).]
The property that the branch locus contains a $k$-rational point is
invariant under $k$-M\"obius equivalence, since
$\operatorname{PGL}_2(k)$ preserves $\mathbf P^1(k)$.
For $\gamma\in E(k)$, translation by $\gamma$ conjugates
$\iota_\varepsilon$ to $\iota_{\varepsilon+2\gamma}$, and translation
by $\varphi(\gamma)$ gives the corresponding conjugacy on $E'$. These
translations descend to $k$-isomorphisms of both quotient lines and
conjugate the lower morphisms. Thus the class depends only on
$\varepsilon\bmod2E(k)$.

Let \(g:E/\langle\iota_\varepsilon\rangle\to
E'/\langle\iota_{\varphi(\varepsilon)}\rangle\) be the lower morphism,
and let $\pi,\pi'$ be the quotient morphisms.
For $P\in E(\bar k)$ and $v=\pi(P)$, multiplicativity of local
ramification indices gives
$e_P(\pi)e_v(g)=e_P(\varphi)e_{\varphi(P)}(\pi')
=e_{\varphi(P)}(\pi')$, since $\varphi$ is unramified. Hence
$\operatorname{Br}(g)\subseteq\operatorname{Br}(\pi')$.

Suppose the inclusion is strict, and choose
$\beta\in\operatorname{Br}(\pi')\setminus\operatorname{Br}(g)$. The
reduced geometric fibre of $\pi'$ over $\beta$ is a singleton
$\{u_1\}$, with $2u_1=\varphi(\varepsilon)$. Choose
$u_0\in\varphi^{-1}(u_1)$. As sets of geometric points,
$\pi^{-1}\bigl(g^{-1}(\beta)\bigr)=(\pi'\varphi)^{-1}(\beta)=u_0+N$.
The right side has $n$ points. Since $\beta\notin\operatorname{Br}(g)$,
the reduced fibre $g^{-1}(\beta)$ has $n$ distinct points. The
separable degree-two morphism $\pi$ has at least one point above each;
therefore it has exactly one, of ramification index $2$. Consequently
all points $u_0+\kappa$, $\kappa\in N$, are fixed by
$\iota_\varepsilon$, so $2(u_0+\kappa)=\varepsilon$ for
$\kappa\in N$. Subtracting the equality for $\kappa=0$ gives $2\kappa=0$ for every
$\kappa\in N$, contrary to the odd order of $N$. Thus
$\operatorname{Br}(g)=\operatorname{Br}(\pi')$.

For a branch point $\beta$ of $\pi'$, its unique geometric preimage
$u$ satisfies $2u=\varphi(\varepsilon)$. Since the fibre is a singleton,
Galois invariance gives
$\beta\in\mathbf P^1(k)\Longleftrightarrow u\in E'(k)$. Hence the branch locus of $g$ contains a $k$-rational point if and only
if $\varphi(\varepsilon)\in2E'(k)$. One implication from
$\varepsilon\in2E(k)$ follows by applying $\varphi$. Conversely, suppose
$\varphi(\varepsilon)=2u$ with $u\in E'(k)$. Let
$\widehat\varphi:E'\to E$ be the \emph{dual isogeny}, characterized by
$\widehat\varphi\varphi=[n]$; see
\cite[Theorem~III.6.1]{Silverman2009ArithmeticEllipticCurves}. It is
defined over $k$ by uniqueness. Thus
$n\varepsilon=2\widehat\varphi(u)\in2E(k)$. Since $n=2m+1$,
$\varepsilon=n\varepsilon-2m\varepsilon\in2E(k)$. The first paragraph now
gives the last assertion.
\end{proof}

\begin{proof}[Proof of Theorem~\ref{thm:intro-dihedral} (3).]
Choose a generator $\xi$ of $N$, and choose $k$-coordinates $x,t$ on
the source and target quotient lines so that $t=f(x)$. Pull $x$ back
to $E$. For $\gamma\in\operatorname{Aut}(E_{\bar k})$, let
$\widetilde\gamma(h)=h\circ\gamma^{-1}$ on $\bar k(E)$, and put $x_j=\widetilde\tau_{j\xi}(x)$ for
$j\in\mathbb Z/n\mathbb Z$. Put $\Gamma=T_N\rtimes\langle\iota_\varepsilon\rangle$. The pullback of
$t$ is $\Gamma$-invariant, and the stabilizer of $x$ in $\Gamma$ is
$\langle\iota_\varepsilon\rangle$. Since
$T_N\cap\langle\iota_\varepsilon\rangle=1$, the functions $x_j$ are
pairwise distinct. They all satisfy $f(x_j)=t$; hence they are exactly
the $n$ roots of $\Phi_f$. The geometric generators act
by
\[
 \widetilde\tau_\xi:x_j\longmapsto x_{j+1},
 \qquad
 \widetilde\iota_\varepsilon:x_j\longmapsto x_{-j}.
\]

The semilinear action $\widetilde\sigma_q$ fixes $x,t$ and satisfies
\[
 \widetilde\sigma_q\widetilde\tau_\alpha
 \widetilde\sigma_q^{-1}=\widetilde\tau_{\sigma_q(\alpha)},
 \qquad
 \widetilde\sigma_q\widetilde\iota_\varepsilon
 \widetilde\sigma_q^{-1}=\widetilde\iota_\varepsilon.
\]
Thus $\widetilde\sigma_q(x_j)=x_{\lambda j}$. The roots $x_j$
generate the splitting field of $\Phi_f$ over $k(t)$, so
$\Omega=k(x_j:j\in\mathbb Z/n\mathbb Z)$. Its constant field
$k_\Omega$ is finite over $k$, and $\operatorname{Gal}(k_\Omega/k)$ is
generated by arithmetic Frobenius. Thus $\widetilde\sigma_q$ restricts to $\bar\sigma_q\in A_f$, whose image
generates $A_f/G_f$, and
$A_f=\langle G_f,\bar\sigma_q\rangle$.

Identifying the sheet $x_j$ with $j\in\mathbb Z/n\mathbb Z$, these
generators act as $j\mapsto j+1$, $j\mapsto-j$, and
$j\mapsto\lambda j$. Hence
\[
\begin{aligned}
 G_f&\cong\{j\mapsto\pm j+b:b\in\mathbb Z/n\mathbb Z\},\\
 A_f&\cong\{j\mapsto aj+b:
 a\in\langle-1,\lambda\rangle,
 b\in\mathbb Z/n\mathbb Z\}.
\end{aligned}
\]
The stabilizers of $0$ are $G_1=\{\pm1\}$ and
$A_1=\langle-1,\lambda\rangle$. For $j\ne0$, the $G_1$-orbit is
$\{j,-j\}$, while the $A_1$-orbit is
$\langle-1,\lambda\rangle j$. These two orbits coincide exactly when
$\lambda j=\pm j$. A nonzero solution of
$(\lambda\mp1)j=0$ in $\mathbb Z/n\mathbb Z$ exists exactly when
$\gcd(n,\lambda\mp1)>1$. Cohen's criterion
\cite[Lemma~2.4(2)]{DingZieve2022LowDegree} now gives the assertion.
\end{proof}

The criterion also shows that no such exceptional map exists
when $3\mid n$. If $n$ is prime, it reduces to
$\lambda\not\equiv\pm1\pmod n$; in particular, for $n=5$ one has
$\lambda=\pm2$. For composite $n$, the latter condition is insufficient:
for example, $n=15$ and $\lambda=2$ give $2\cdot5=-5$ in
$\mathbb Z/15\mathbb Z$.

Retain the defect notation of Theorem~\ref{thm:intro-signed-descent}.
If a half \(\omega\) of \(\varepsilon\) is replaced by \(\omega+T\),
with \(T\in E[2]\), then
\(\omega^{\sigma_q}-\omega\) changes by \((\sigma_q-1)T\); hence the
defect is well defined. Moreover, $\mathfrak d_E:E(k)\to\mathfrak D(E)$ is a homomorphism,
since halves add.

\begin{lem}[The defect isomorphism]
\label{lem:defect-isomorphism}
For every elliptic curve $E/k$, the Frobenius defect induces an
isomorphism \(E(k)/2E(k)\xrightarrow{\sim}\mathfrak D(E)
=E[2]/(\sigma_q-1)E[2]\), and we have
\(\mathfrak d_E(P)=0\Longleftrightarrow P\in2E(k)\).
\end{lem}

\begin{proof}
If $P=2Q$ with $Q\in E(k)$, then $Q$ is a rational half of $P$, so
$\mathfrak d_E(P)=0$. Conversely, if $2\omega=P$ and the defect is zero, then
$\omega^{\sigma_q}-\omega=(\sigma_q-1)T$ for some $T\in E[2]$. Thus $\omega-T\in E(k)$ and
$P=2(\omega-T)$. Hence the defect induces an injection
$E(k)/2E(k)\hookrightarrow\mathfrak D(E)$. The two groups have the
same order: the kernel--cokernel count for $[2]$ on $E(k)$ gives
$|E(k)/2E(k)|=|E[2](k)|$, while rank--nullity for $\sigma_q-1$ on
the finite $\mathbb F_2$-vector space $E[2]\subset E(\bar k)$ gives
$|\mathfrak D(E)|=|E[2](k)|$.
\end{proof}

\begin{lem}[Lifting an equivalence to the geometric Galois closure]
\label{lem:galois-closure-lift}
For $i=1,2$, let $\pi_i:X_i\to Y_i$ be finite separable covers over
$\bar k$, with geometric Galois closures $\widetilde X_i\to Y_i$,
deck groups $\Gamma_i$, and point stabilizers $H_i$. An isomorphism of
covering diagrams
\[
\xymatrix@R=1.25em{
 X_1\ar[r]^u\ar[d]_{\pi_1}&X_2\ar[d]^{\pi_2}\\
 Y_1\ar[r]^v&Y_2
}
\]
lifts to $h:\widetilde X_1\to\widetilde X_2$ with
$h\Gamma_1h^{-1}=\Gamma_2$ and $hH_1h^{-1}=H_2$. The lifts inducing
$u$ form a left $H_2$-torsor. If the square descends to $k$, then
$({}^\sigma h)h^{-1}\in H_2$ for every
$\sigma\in\operatorname{Gal}(\bar k/k)$, where
${}^\sigma h:=\sigma_{\widetilde X_2}h\sigma_{\widetilde X_1}^{-1}$.
\end{lem}

\begin{proof}
On function fields, $(u,v)$ identifies the two intermediate extensions
and their embeddings into an algebraic closure; hence it identifies the
fields generated by all conjugates, namely their normal closures.
Galois correspondence gives the two conjugacy equalities. If two lifts
induce $u$, their quotient is a deck transformation of
$\widetilde X_2\to X_2$, hence lies in $H_2$; left composition by any element of $H_2$
produces another lift of $u$. When the square is defined over $k$, $h$ and ${}^\sigma h$
induce the same $u$, proving the last assertion.
\end{proof}

\begin{proof}[Proof of Theorem~\ref{thm:intro-signed-descent}.]
Put $H_i=\langle\iota_{\varepsilon_i}\rangle$ and
$\Gamma_i=T_{N_i}\rtimes H_i\cong D_n$. As in the converse part of
Theorem~\ref{thm:intro-dihedral}\textup{(1)}, each $H_i$ is core-free;
hence the geometric Galois closures are
$E_{i,\bar k}\to E_{i,\bar k}/\Gamma_i$.

By Lemma~\ref{lem:galois-closure-lift}, a \(k\)-M\"obius equivalence of the
lower morphisms lifts to an isomorphism
$h:E_{1,\bar k}\to E_{2,\bar k}$ with
$h\Gamma_1h^{-1}=\Gamma_2$ and $hH_1h^{-1}=H_2$. Since
$[\tau_a,\iota_{\varepsilon_i}]=\tau_{2a}$ and $2N_i=N_i$, one has
$[\Gamma_i,\Gamma_i]=T_{N_i}$; hence $hT_{N_1}h^{-1}=T_{N_2}$. Put
$R=h(O_{E_1})$ and $\alpha=\tau_{-R}h$. Then $\alpha$ is an
origin-preserving isomorphism, $\alpha(N_1)=N_2$, and
\begin{equation}
 h=\tau_R\alpha,
 \qquad \varepsilon_2=\alpha(\varepsilon_1)+2R.
\label{eq:translation-data}
\end{equation}
The final assertion of Lemma~\ref{lem:galois-closure-lift} gives
$c=({}^{\sigma_q}h)h^{-1}\in H_2$. Since
$H_2=\{1,\iota_{\varepsilon_2}\}$, the two possibilities give
\[
 {}^{\sigma_q}\alpha=\pm\alpha,
 \qquad
 R^{\sigma_q}=R\ \text{or}\ \varepsilon_2-R,
\]
with corresponding signs.

Choose $2\omega=\varepsilon_1$ and put $P=h(\omega)$, so
$2P=\varepsilon_2$. In the positive case,
$P^{\sigma_q}-P=\alpha(\omega^{\sigma_q}-\omega)$. In the negative
case, using \eqref{eq:translation-data},
\[
 P^{\sigma_q}-P
 =\varepsilon_2-2R-\alpha(\omega^{\sigma_q}+\omega)
 =-\alpha(\omega^{\sigma_q}-\omega)
 =\alpha(\omega^{\sigma_q}-\omega),
\]
since $\omega^{\sigma_q}-\omega\in E_1[2]$. This proves necessity.

Conversely, suppose \eqref{eq:signed-data} holds. Choose halves
$2\omega=\varepsilon_1$ and $2P=\varepsilon_2$. Equality of the defect
classes means that there is $T\in E_2[2]$ such that
\[
 (P^{\sigma_q}-P)-\alpha(\omega^{\sigma_q}-\omega)
   =(\sigma_q-1)T.
\]
Replacing $P$ by $P-T$ therefore gives
\begin{equation}
 P^{\sigma_q}-P=\alpha(\omega^{\sigma_q}-\omega).
\label{eq:defect-match}
\end{equation}
Set $R=P-\alpha(\omega)$ and $h=\tau_R\alpha$. Then
\eqref{eq:translation-data} holds and $h$ identifies $H_1,T_{N_1}$
with $H_2,T_{N_2}$. Equation~\eqref{eq:defect-match} gives
$R^{\sigma_q}=R$ in the positive case and
$R^{\sigma_q}=\varepsilon_2-R$ in the negative case. Hence
${}^{\sigma_q}h=h$ or
${}^{\sigma_q}h=\iota_{\varepsilon_2}h$.

After quotienting by $H_2$, either relation becomes Frobenius-invariant,
so the induced source-quotient isomorphism descends to $k$. The map $h$
also induces $\bar h:E_1/N_1\to E_2/N_2$; in the two cases,
${}^{\sigma_q}\bar h=\bar h$ or
${}^{\sigma_q}\bar h=\iota_{\varphi_2(\varepsilon_2)}\bar h$.
Quotienting by the target involution therefore gives a second
$k$-isomorphism. These isomorphisms identify the lower diagrams, proving
$\mathcal M(\mathcal D_1)=\mathcal M(\mathcal D_2)$.
\end{proof}

Changing the halves changes only the representatives of the defect classes.
Replacing the lift changes it by left composition with $H_2$, while changing
the origin correction changes only the auxiliary translation. Hence none of
these choices affects the criterion.

\section{Cyclic monodromy}
\label{sec:cyclic}

We first treat cyclic monodromy. The tame case is Kummer--R\'edei, while
the wild degree-five case is additive.
\begin{lem}
\label{cyclic-basic}
Let $k$ be perfect and $f\in k(X)$ have degree $n\geqslant2$.
\begin{enumerate}
\item[(1)] If $f^{-1}(\beta)=\{\alpha\}$, then
$\alpha\in\mathbf P^1(k(\beta))$, where $k(\infty)=k$.
\item[(2)] The map $f$ has a totally ramified branch point in
$\mathbf P^1(k)$ if and only if it is $k$-M\"obius equivalent to a
polynomial.
\item[(3)] The map $f$ has two distinct totally ramified branch points in
$\mathbf P^1(k)$ if and only if $f\sim_kX^n$.
\item[(4)] If $(n,\operatorname{char}k)=1$ and $f$ has two totally
ramified branch points over $\bar k$, then it has no others.
\end{enumerate}
\end{lem}

\begin{proof}
For (1), put $K=k(\beta)$. Every
$\sigma\in\operatorname{Gal}(\bar k/K)$ fixes the singleton fibre and
hence fixes $\alpha$. Since $K$ is perfect,
$\alpha\in\mathbf P^1(K)$. For (2), part (1) puts the unique preimage of
a rational totally ramified value in $\mathbf P^1(k)$. Source and target
M\"obius transformations sending these two points to $\infty$ make $f$
a polynomial; the converse is immediate. For (3), send the two branch
values and their preimages to $0,\infty$. The resulting function has one
zero and one pole, hence is $cX^n$, and the scalar is removed on the
left. For (4), apply (3) over $\bar k$ and note that the tame morphism
$X^n$ branches only at $0,\infty$.
\end{proof}

Let $k=\mathbb F_q$, and let
$\sigma_q\in\operatorname{Gal}(\bar k/k)$ be the \emph{arithmetic
Frobenius} $a\mapsto a^q$, extending coefficientwise to $\bar k(X)$ by
fixing $X$. For an integer $m\geqslant1$ prime to
$\operatorname{char}k$, write
\[
 \boldsymbol{\mu}_m:=\{\zeta\in\bar k^\times:\zeta^m=1\}.
\]
For $\delta\in\mathbb F_{q^2}\setminus\mathbb F_q$, put
\[
 \nu_\delta(X)=\frac{X-\delta^q}{X-\delta},
 \qquad R_{n,\delta}:=\nu_\delta^{-1}\circ X^n\circ\nu_\delta.
\]
With $\iota(X):=1/X$, one has
$\sigma_q(\nu_\delta)=\iota\nu_\delta$ and
$\sigma_q(\nu_\delta^{-1})=\nu_\delta^{-1}\iota$; hence
$R_{n,\delta}\in k(X)$. Moreover $\nu_\delta$ maps
$\mathbf P^1(k)$ into the norm-one subgroup
$\boldsymbol{\mu}_{q+1}\subset\mathbb F_{q^2}^{\times}$.
It is injective, and both sets have $q+1$ elements, so it is bijective.
Thus $R_{n,\delta}$ permutes $\mathbf P^1(k)$ exactly when
$(n,q+1)=1$.

\begin{thm}
\label{cyclic-classification}
Let $k=\mathbb F_q$, and let $f\in k(X)$ have degree $n\geqslant2$, with
$(n,\operatorname{char}k)=1$ and $G_f\cong C_n$.
\begin{enumerate}
\item[(1)] If $f$ has a branch point in $\mathbf P^1(k)$, then
$f\sim_kX^n$, and it is exceptional if and only if $(n,q-1)=1$.
\item[(2)] If it has none, then $f\sim_kR_{n,\delta}$ for every
$\delta\in\mathbb F_{q^2}\setminus\mathbb F_q$, and it is exceptional
if and only if $(n,q+1)=1$.
\end{enumerate}
Conversely, these forms have geometric monodromy group isomorphic to $C_n$; the branch
points of $X^n$ are rational, whereas those of $R_{n,\delta}$ are not.
\end{thm}

\begin{proof}
Proposition~\ref{cyclic-ramification} shows that
$f$ has exactly two totally ramified branch points; they form a
Frobenius-stable pair. If one is rational, both
are, and Lemma~\ref{cyclic-basic}(3) gives $f\sim_kX^n$.
Now $X^n$ permutes $\mathbf P^1(\mathbb F_{q^m})$ exactly when
$(n,q^m-1)=1$. If a prime $\ell\mid n$ divides $q-1$, this never holds.
Conversely, if $(n,q-1)=1$, put
$M=\operatorname{lcm}_{\ell\mid n}\operatorname{ord}_\ell(q)$; every
$m\equiv1\pmod M$ works.

Suppose neither branch point is rational. Write them as
$\beta,\beta^q$. Lemma~\ref{cyclic-basic}(1) gives their unique preimages
$\alpha,\alpha^q\in\mathbf P^1(\mathbb F_{q^2})$, and none of these
four points is rational. Since every element of
$\mathbb F_{q^2}\setminus\mathbb F_q$ is $a\gamma+b$ for any fixed
$\gamma$ in that set and suitable $a\in k^\times,b\in k$, independent
source and target changes let us take
$\alpha=\beta=\delta$ and $\alpha^q=\beta^q=\delta^q$. Then
$h=\nu_\delta\circ f\circ\nu_\delta^{-1}=cX^n$ for some
$c\in\mathbb F_{q^2}^\times$. Frobenius gives
$c^qX^n=\iota h\iota=c^{-1}X^n$, so $c^{q+1}=1$. Hence
$L=\nu_\delta^{-1}\circ(cX)\circ\nu_\delta$ is fixed by Frobenius,
and $f=L\circ R_{n,\delta}$ with $L\in\operatorname{PGL}_2(k)$.

For odd $m$, $\delta^{q^m}=\delta^q$, and $\nu_\delta$ identifies
$\mathbf P^1(\mathbb F_{q^m})$ with the norm-one subgroup
$\{z\in\mathbb F_{q^{2m}}^\times:z^{q^m+1}=1\}$, of order $q^m+1$.
Thus $R_{n,\delta}$ is a permutation exactly when $(n,q^m+1)=1$.
If some $\ell\mid n$ divides $q+1$, this fails for odd $m$; for even
$m$, $R_{n,\delta}$ is conjugate to $X^n$ and $\ell\mid q^m-1$.
Conversely, if $(n,q+1)=1$, every
$m\equiv1\pmod{2M}$ works, with $M$ as above.

Finally, $\bar k(X)/\bar k(X^n)$ is cyclic of degree $n$, generated by
$X\mapsto\zeta X$, and branches only at $0,\infty$. The assertions for
$R_{n,\delta}$ follow by geometric conjugacy.
\end{proof}

A polynomial in characteristic $p$ is \emph{additive} if it has the
form $\sum_i a_iX^{p^i}$; it acts $\mathbb F_p$-linearly on every
extension field.
\begin{prop}
\label{additive-char5}
Let $k=\mathbb F_q$ have characteristic $5$. If $f\in k(X)$ is
separable and exceptional of degree $5$, with $G_f\cong C_5$, then
$f\sim_kX^5-aX$ for some
$a\in k^\times\setminus(k^\times)^4$. Conversely, every such polynomial
is separable and exceptional, has geometric monodromy group isomorphic to $C_5$, and has
$\infty$ as its unique, totally ramified branch point.
\end{prop}

\begin{proof}
By Ding--Zieve \cite[Proposition~4.3]{DingZieve2022LowDegree}, $f$ is
\(k\)-M\"obius equivalent to an additive polynomial. Degree and separability give
$AX^5+BX$ with $AB\ne0$, hence the asserted form. The map
$g=X^5-aX$ is bijective on a finite extension precisely when $X^4=a$
has no nonzero solution. Thus it permutes $k$ exactly when
$a\notin(k^\times)^4$. Write $k^\times=\langle\gamma\rangle$ and $a=\gamma^r$. Relative to a
generator of $\mathbb F_{q^m}^\times$, the exponent of $a$ is
$ur(q^m-1)/(q-1)$ for some $u$ coprime to $q-1$. Since
$q\equiv1\pmod4$, both $u$ and $(q^m-1)/(q-1)$ are odd when $m$ is odd.
Thus $a$ is a fourth power over $\mathbb F_{q^m}$ exactly when $4\mid r$.
Hence a non-fourth-power remains one in every odd-degree extension. If
$a$ is a fourth power in $k$, a nonzero $k$-root prevents
permutation on every extension.

Since $g'=-a\ne0$, the roots of $g$ are five distinct points forming a
one-dimensional $\mathbb F_5$-space. Their translations fix $g$, so
$\bar k(X)/\bar k(g(X))$ is Galois with group $C_5$. At a finite point
$\alpha$, $g(\alpha+u)-g(\alpha)=u^5-au$ has order $1$ in the local
parameter $u$; at infinity, $1/g(1/u)=u^5/(1-au^4)$ has order $5$.
Thus only infinity ramifies, totally.
\end{proof}

\section{Dihedral maps with a rational branch point}
\label{sec:rational-branch}

We treat the dihedral cases with a rational branch point. Total
ramification gives the classical exceptional-polynomial families. For
reflection inertia, Theorem~\ref{thm:intro-dihedral} reduces the shifted
quotient to the ordinary quotient by negation. For prime $\ell$,
Bisson--Tibouchi treat the unshifted case in characteristic different from
$2$ and $3$; their algorithms in \S\S3--4 additionally assume
$\ell\ne\operatorname{char}k$
\cite[Theorem~2.1 and \S\S3--4]{BissonTibouchi2018Isogenies}. Under these
restrictions, their construction and permutation-kernel criterion are the
classical counterparts of the ones used here, with kernel-polynomial
formulas as in Kohel \cite[\S2.4]{Kohel1996EndomorphismRings}. The following
corollary records the form needed below. It is a converse for
rational-branch reflection covers, valid in arbitrary odd degree and in all
characteristics, including reduced kernels when the characteristic divides
the degree. We include a short derivation to make the last point explicit.

\begin{cor}[Rational-branch kernel-polynomial form]
\label{rational-branch-kernel-form}
Let $n=2m+1\geqslant3$ be odd, let $k=\mathbb F_q$, and let
$f\in k(X)$ be separable of degree $n$. Suppose that
$G_f\cong D_n$, every nontrivial inertia group of $f$ is generated by a
reflection, and
\[
 \operatorname{Br}(f)\cap\mathbf P^1(k)\ne\varnothing.
\]
Then $f\sim_k g_{E,Q}$ for parameters of the following form.

Choose coefficients $a_1,a_2,a_3,a_4,a_6\in k$ defining a nonsingular
\emph{long Weierstrass equation}
\[
\begin{aligned}
 E:\quad &y^2+a_1xy+a_3y=x^3+a_2x^2+a_4x+a_6,\\
 b_2&:=a_1^2+4a_2,\qquad
 b_4:=2a_4+a_1a_3,\qquad
 b_6:=a_3^2+4a_6,\\
 b_8&:=a_1^2a_6+4a_2a_6-a_1a_3a_4+a_2a_3^2-a_4^2,\\
 \Delta_E&:=-b_2^2b_8-8b_4^3-27b_6^2+9b_2b_4b_6\ne0.
\end{aligned}
\]
\[
\hbox to \displaywidth{%
  \rlap{Set}\hfil
  $\displaystyle \mathcal U_E(X):=4X^3+b_2X^2+2b_4X+b_6$.
  \hfil
}
\]
\[
\hbox to \displaywidth{%
  \rlap{Let}\hfil
  $\displaystyle Q(X)=X^m-s_1X^{m-1}+\cdots+(-1)^ms_m\in k[X]$
  \hfil
}
\]
\[
\hbox to \displaywidth{%
  \rlap{be monic and satisfy}\hfil
  $\displaystyle \gcd(Q,Q')=1,\qquad \gcd(Q,\mathcal U_E)=1$.
  \hfil
}
\]
\[
\hbox to \displaywidth{%
  \rlap{The associated set}\hfil
  $\displaystyle N_Q:=\{O_E\}\cup\{P\in E(\bar k):Q(x(P))=0\}$
  \hfil
}
\]
is required to be a cyclic subgroup of $E(\bar k)$ of order $n$.
Equivalently, there is a point $\xi\in E(\bar k)$ of exact order $n$
such that
\[
 N_Q=\langle\xi\rangle,
 \qquad
 Q(X)=\prod_{j=1}^{m}\bigl(X-x([j]\xi)\bigr).
\]
For these parameters, define
\begin{equation}
\label{eq:kernel-polynomial-isogeny}
 g_{E,Q}(X):=nX-2s_1-\mathcal U_E(X)\left(\frac{Q'}{Q}\right)'
 -\bigl(6X^2+b_2X+b_4\bigr)\frac{Q'}{Q}.
\end{equation}
Here primes denote formal derivatives, and $n$ is viewed as an element
of $k$; in particular, the term $nX$ may vanish when
$\operatorname{char}k\mid n$.

Conversely, every pair $(E,Q)$ satisfying the preceding conditions
yields via \eqref{eq:kernel-polynomial-isogeny} a separable
rational function $g_{E,Q}\in k(X)$ of degree $n$, with
\(G_{g_{E,Q}}\cong D_n\); every nontrivial inertia group is generated
by a reflection, and
$\infty$ is a $k$-rational branch point. If
$\operatorname{char}k\mid n$, the existence of $N_Q$ forces $E$ to be
ordinary.

Let $\lambda\in(\mathbb Z/n\mathbb Z)^\times$ be characterized by
\(\sigma_q(P)=[\lambda]P\) for \(P\in N_Q\).
Then the following are equivalent:
\begin{enumerate}
\item[\textup{(i)}] $g_{E,Q}$ is exceptional over $k$;
\item[\textup{(ii)}]
$\gcd(n,\lambda-1)=\gcd(n,\lambda+1)=1$;
\item[\textup{(iii)}] $Q$ has no root in $k$;
\item[\textup{(iv)}] $\gcd(Q,X^q-X)=1$.
\end{enumerate}
\end{cor}

\begin{proof}
Theorem~\ref{thm:intro-dihedral}\textup{(1)} gives an admissible isogeny datum
$(E,N,\varepsilon)$ of degree $n$ whose lower morphism represents the
$k$-M\"obius class of $f$. By Theorem~\ref{thm:intro-dihedral}\textup{(2)},
the rational branch point is equivalent to $\varepsilon\in2E(k)$, and
then \(\mathcal M(E,N,\varepsilon)=\mathcal M(E,N,O_E)\).
Thus translations defined over $k$ identify the two shifted Kummer
quotients with the ordinary quotients by negation. Choose a long
Weierstrass equation for $E$, let $\varphi:E\to E'=E/N$ be the quotient
isogeny, and replace $E'$ by the normalized V\'elu quotient model. With
$x,x'$ denoting the corresponding Weierstrass coordinates, the lower
morphism satisfies \(x'\circ\varphi=g\circ x\).

Since $n=2m+1$, the nonzero points of $N$ form $m$ pairs
$\{\pm T\}$. Its monic \emph{kernel polynomial} is
$Q_N(X)=\prod_{T\in(N\setminus\{O_E\})/\{\pm1\}}(X-x(T))\in k[X]$.
It is squarefree because $N$ is reduced. Moreover,
\((2y+a_1x+a_3)^2=\mathcal U_E(x)\) on $E$, so the odd order of $N$ gives
$\gcd(Q_N,\mathcal U_E)=1$. Hence $Q_N$ has precisely the stated
properties and $N_{Q_N}=N$.

In the classical prime-to-characteristic setting,
Formula~\eqref{eq:kernel-polynomial-isogeny} is the usual
kernel-polynomial form of V\'elu's quotient map; see \S2.4 of
\cite{Kohel1996EndomorphismRings}. For the characteristic-safe argument
needed here, let $P$ be the generic point of $E$ and define
\[
 \widetilde x(P):=x(P)+\sum_{0\ne T\in N}\bigl(x(P+T)-x(T)\bigr)
 =\sum_{T\in N}x(P+T)-\sum_{0\ne T\in N}x(T).
\]
Translation by an element of $N$ permutes the first sum, so
$\widetilde x$ is $N$-invariant. Frobenius also permutes $N$, while
$x$ is defined over $k$; hence $\widetilde x$ is defined over $k$ and
lies in $k(E/N)$. Because $N$ is reduced, the $n$ points $-T$,
$T\in N$, are pairwise distinct, and the pole divisor of
$\widetilde x$ is $2\sum_{T\in N}(-T)$, the pullback of
$2(O_{E/N})$. On the elliptic quotient $E/N$, Riemann--Roch gives
$\ell(2O_{E/N})=2$. We therefore choose the quotient Weierstrass model
so that the descended function represented by $\widetilde x$ is its
$x$-coordinate; with this choice, the formula below is the normalized
quotient coordinate. This argument uses only the reduced kernel and is
valid in arbitrary characteristic.

Now put $x(P)=X$, and for $0\ne T\in N$ write $u=x(T)$. The addition law
on the long Weierstrass model gives the pair identity
\[
 x(P+T)+x(P-T)-2u
 =2X-2u+\frac{\mathcal U_E(X)}{(X-u)^2}
   -\frac{6X^2+b_2X+b_4}{X-u}.
\]
If $u_1,\ldots,u_m$ are the roots of $Q_N$, then
\[
 \frac{Q_N'}{Q_N}=\sum_{i=1}^m\frac1{X-u_i},
 \qquad
 \left(\frac{Q_N'}{Q_N}\right)'
 =-\sum_{i=1}^m\frac1{(X-u_i)^2}.
\]
Summing the pair identity over the $m$ pairs $\{\pm T\}$ in the
preceding formula for $\widetilde x$ gives
\eqref{eq:kernel-polynomial-isogeny} directly. In particular, the
derivation uses neither the multiplication map $[n]$ nor division by $n$;
therefore it remains valid when $\operatorname{char}k\mid n$, provided the
kernel is reduced.

Conversely, let $(E,Q)$ satisfy the stated conditions. Then $N_Q$ is
a reduced cyclic subgroup of order $n$. Since $Q\in k[X]$, Frobenius
stabilizes $N_Q$. Silverman
\cite[Proposition~III.4.12 and
Remark~III.4.13.2]{Silverman2009ArithmeticEllipticCurves} therefore
gives a separable quotient isogeny
$\varphi:E\longrightarrow E/N_Q$ defined over $k$, and \eqref{eq:kernel-polynomial-isogeny} is its
normalized $x$-coordinate map. Applying
Theorem~\ref{thm:intro-dihedral}\textup{(1)} to $(E,N_Q,O_E)$ gives the
assertions about degree, separability, monodromy, and inertia.
The proof of Theorem~\ref{thm:intro-dihedral}\textup{(2)} identifies the branch locus
with that of the target Kummer quotient; hence the image of the origin,
which is $\infty$ in the chosen coordinate, is a $k$-rational branch
point. If $p=\operatorname{char}k$ divides $n$, then $N_Q$ contains a
nonzero geometric point of order $p$, so $E$ is ordinary by
\cite[Corollary~III.6.4(c)]{Silverman2009ArithmeticEllipticCurves}.

It remains to prove the last equivalences. Theorem~\ref{thm:intro-dihedral}\textup{(3)}
gives \textup{(i)}$\Longleftrightarrow$\textup{(ii)}. For
$0\ne P\in N_Q$, the two points of $E(\bar k)$ above $x(P)$ are
$P$ and $-P$; hence
$x(P)\in k\Longleftrightarrow\sigma_q(P)=P\text{ or }-P$.
Since $\sigma_q(P)=[\lambda]P$, such a nonzero point exists precisely
when one of $(\lambda-1)P=O_E$ and $(\lambda+1)P=O_E$ has a nonzero
solution in the cyclic group $N_Q$. This is equivalent to \(\gcd(n,\lambda-1)>1\) or
\(\gcd(n,\lambda+1)>1\). Thus \textup{(ii)} is equivalent to \textup{(iii)}. Finally,
\textup{(iii)} and \textup{(iv)} are equivalent because $Q\in k[X]$.
\end{proof}

We now return to degree five. By Corollary~\ref{total-branch-polynomial},
a dihedral map with total ramification is $k$-M\"obius equivalent to a
polynomial, so this case reduces to the classical exceptional-polynomial
classification. Two polynomials are \emph{affinely equivalent} if they
differ by left and right composition with degree-one polynomials. A
polynomial is \emph{indecomposable} if it is not a composition of two
polynomials of degree greater than one.

\begin{thm}
\label{poly-dihedral}
Let \(k=\mathbb{F}_q\), and let \(f\in k(X)\) be separable, exceptional,
of degree \(5\), with \(G_f\cong D_{5}\). If \(f\) has a totally ramified
branch point, then \(f\sim_{k} g\), where either
\[
 g=D_{5}(X,a),\qquad
 a\in k^\times,\quad \operatorname{char}k\ne5,\quad
 q\equiv\pm2\pmod5,
\]
\[
\hbox to \displaywidth{%
  \rlap{or, when $\operatorname{char}k=5$,}\hfil
  $\displaystyle g=X(X^2-a)^2,\qquad a\in k^\times\text{ nonsquare}$.
  \hfil
}
\]
Conversely, the displayed functions are separable and exceptional
over \(k\), have geometric monodromy group isomorphic to \(D_{5}\),
and have \(\infty\) as a totally ramified branch point.
\end{thm}

\begin{proof}
Corollary~\ref{total-branch-polynomial} gives \(f\sim_{k} h\), where
\(h\) is an exceptional polynomial of prime degree \(5\), hence is
indecomposable. Since $k$-M\"obius equivalence preserves geometric
monodromy, $G_h\cong G_f\cong D_5$; in particular, the relevant geometric
group is solvable. If \(\operatorname{char}k\ne5\), M\"uller's
prime-degree solvable classification
\cite[Appendix and Theorem~4]{Muller1997Schur} therefore shows that, up to
affine equivalence, \(h\) is a monomial or a Dickson polynomial. The former
has cyclic geometric monodromy; thus
\(h\sim_{k}D_{5}(X,a)\), \(a\ne0\). By
\cite[Theorem~7.16]{LidlNiederreiter1997FiniteFields}, this polynomial
permutes \(\mathbb{F}_{q^m}\) exactly when
\((5,q^{2m}-1)=1\), which holds for infinitely many \(m\) precisely
when \(q\equiv\pm2\pmod5\).

Suppose \(\operatorname{char}k=5\). The classification
\cite[Theorem~8.1 and (8.2)]{FriedGuralnickSaxl1993} gives, up to affine equivalence,
$X\bigl(X^{4/s}-a\bigr)^s$, where $s\mid4$ and $a\in k^\times$.
For \(s=1\) the geometric monodromy group is isomorphic to \(C_5\),
by Proposition~\ref{additive-char5}. For \(s=4\),
\cite[condition~(8.3)(i)]{FriedGuralnickSaxl1993} would require the linear polynomial
\(X-a\) to have no root in \(k\). Hence \(s=2\), and the same condition
says exactly that \(a\) is nonsquare.

Conversely, \cite[Theorem~8.1 and condition~(8.3)(i)]{FriedGuralnickSaxl1993}
makes $g=X(X^2-a)^2$ exceptional when $a$ is nonsquare. In
characteristic $5$,
$g'=-a(X^2-a)\ne0$, so $g$ is separable. Its fibre over $0$ has
partition $\{\!\{2,2,1\}\!\}$. The points of index $2$ are tamely
ramified, so inertia contains an involution of cycle type $2^2 1$.
Since $C_5$ has no involution, Proposition~\ref{monodromy-reduction}
gives $G_g\cong D_5$. For the Dickson
family, put \(x=Y+a/Y\). The identity
\(D_{5}(x,a)=Y^5+(a/Y)^5\) shows that \(Y\mapsto\zeta Y\), with \(\zeta\) a primitive fifth
root of unity, and \(Y\mapsto a/Y\) generate its
dihedral fixing group of order \(10\). Since $p\ne5$, the derivative of $Y^5+a^5Y^{-5}$ is nonzero. Thus
the degree-ten extension is separable, and the displayed dihedral group
is its full Galois group. The latter involution has fixed field
\(\bar{k}(x)\) and trivial core in \(D_{5}\); hence
\(\bar{k}(Y)\) is the Galois closure of
\(\bar{k}(x)/\bar{k}(D_{5}(x,a))\). The Dickson criterion
\cite[Theorem~7.16]{LidlNiederreiter1997FiniteFields} gives exceptionality. Both
families are separable polynomials of degree \(5\), so \(\infty\) is
totally ramified.
\end{proof}

Henceforth let $f\in k(X)$ be separable and exceptional of degree $5$,
with $G_f\cong D_5$, no totally ramified branch point, and a
$k$-rational branch point. Proposition~\ref{dihedral-ramification}
shows that every branch point is of reflection type. Applying
Corollary~\ref{rational-branch-kernel-form} with $n=5$, we obtain an
elliptic curve $E/k$ and a Frobenius-stable cyclic subgroup
\[
 N=\langle\xi\rangle=\{O_E,\pm\xi,\pm2\xi\}
\]
such that $f\sim_k g_{E,Q_N}$. Since $f$ is exceptional,
Theorem~\ref{thm:intro-dihedral}\textup{(3)} gives
$\sigma_q(\xi)\in\{\pm2\xi\}$. If $x$ is a $k$-Weierstrass
coordinate on $E$ and $\theta=x(\xi)$, then
\[
 Q_N(X)=
 \prod_{T\in(N\setminus\{O_E\})/\{\pm1\}}(X-x(T))
 =(X-\theta)(X-x(2\xi)).
\]

\begin{prop}
\label{rational-branch-reduction}
With the preceding notation,
\[
 \theta^q=x(2\xi),\qquad
 \theta^{q^2}=\theta,\qquad
 \theta^q\ne\theta.
\]
Consequently $Q_N=(X-\theta)(X-\theta^q)$ is irreducible over $k$.
\end{prop}

\begin{proof}
Since $\sigma_q(\xi)\in\{\pm2\xi\}$,
$\theta^q=x(\sigma_q(\xi))=x(2\xi)$. Moreover $\sigma_q^2(\xi)=\sigma_q(\pm2\xi)=(\pm2)^2\xi
=4\xi=-\xi$, whence $\theta^{q^2}=\theta$.
Equality \(\theta^q=\theta\) would give \(2\xi=\pm\xi\), impossible
for a point of order \(5\).
\end{proof}

\begin{lem}[The order-five test]
\label{lem:order-five-test}
Let $E/k$ be an elliptic curve over $k=\mathbb F_q$. If a
Frobenius-stable cyclic subgroup $N\subset E(\bar k)$ of order $5$ has
multiplier $\pm2$, then $N\cap E(k)=\{O_E\}$. Moreover, let $x$ be a
Weierstrass coordinate and let $\xi\in E(\bar k)$ be affine. Put
$t:=x(\xi)$. If
\[
 t\in\mathbb F_{q^2}\setminus k,
 \qquad x(2\xi)=t^q,
\]
then $\sigma_q(\xi)\in\{\pm2\xi\}$ and $\xi$ has exact order $5$;
in particular, $\langle\xi\rangle$ is Frobenius-stable.
\end{lem}

\begin{proof}
Multiplication by $\pm2$ has no nonzero fixed point on a cyclic group
of order $5$, proving the first assertion. For the second,
$x(\sigma_q\xi)=x(2\xi)$ gives $\sigma_q(\xi)=\pm2\xi$. Hence
$\sigma_q^2(\xi)=4\xi$, while $t^{q^2}=t$ gives
$x(4\xi)=x(\xi)$ and thus $4\xi=\pm\xi$. The plus sign would give
$x(2\xi)=x(-\xi)=t$, contrary to $t^q\ne t$. Therefore
$4\xi=-\xi$, so $\xi$ has exact order $5$.
\end{proof}

\begin{lem}[Marked Weierstrass changes]
\label{lem:marked-Weierstrass}
Let $\alpha:E_{\bar k}\to E'_{\bar k}$ be an origin-preserving
isomorphism between long Weierstrass models. Then
\[
 x=u_0^2x'+\rho,
 \qquad
 y=u_0^3y'+u_0^2\nu x'+\tau
\]
for some $u_0\in\bar k^\times$ and $\rho,\nu,\tau\in\bar k$. If
$\alpha(N)=N'$ for reduced odd-order subgroups, the affine map
$v\mapsto u_0^2v+\rho$ carries the roots of the kernel polynomial
of $N'$ onto those of the kernel polynomial of $N$. Moreover:
\begin{enumerate}
\item[(1)] in odd characteristic, if both models have $a_1=a_3=0$,
then $\nu=\tau=0$;
\item[(2)] in characteristic $2$, if $a_1=a_1'=1$, then $u_0=1$,
and if also $a_3=a_3'=0$, then $\rho=0$.
\end{enumerate}
\end{lem}

\begin{proof}
The first assertion and the coefficient transformations are those of
\cite[Table~3.1, p.~45]{Silverman2009ArithmeticEllipticCurves}. The
statement about kernel roots follows from the transformation of the
$x$-coordinate. The two special cases follow directly from the
transformation formulas for $a_1$ and $a_3$.
\end{proof}

For $n=5$, formula~\eqref{eq:kernel-polynomial-isogeny} reads
\[
 g(X)=5X-2s_1
 -(4X^3+b_2X^2+2b_4X+b_6)
 \left(\frac{Q_N'}{Q_N}\right)'
 -(6X^2+b_2X+b_4)\frac{Q_N'}{Q_N},
\]
where $Q_N=X^2-s_1X+s_2$. We now reduce this expression to explicit
normal forms in odd and even characteristic.

\begin{thm}
\label{rational-reflection-odd}
Let \(k=\mathbb{F}_q\) have odd characteristic. Let \(f\in k(X)\)
be separable and exceptional of degree \(5\), with \(G_f\cong D_{5}\),
no totally ramified branch point, and at least one \(k\)-rational
branch point. Then
\[
 f\sim_k\mathcal F_{r,a,d}=
 X+\frac{aX^3+(3d/r)X^2+(3ar-32r^2)X+d}{(X^2-r)^2},
\]
where the triple $(r,a,d)\in k^3$ satisfies
\[
 r\in k^\times\text{ nonsquare},\qquad
 a\ne52r,\qquad
 d^2=a^2r^3-16ar^4+128r^5.
\]
Conversely, for every such triple $(r,a,d)$, one has
\[
 X^2-r\nmid aX^3+(3d/r)X^2+(3ar-32r^2)X+d,
\]
and $\mathcal F_{r,a,d}$ is separable and exceptional of degree $5$
over $k$, has geometric monodromy group isomorphic to $D_5$, has no
total ramification, and has $k$-rational reflection branch point
$\infty$.
\end{thm}

\begin{proof}
Use Proposition~\ref{rational-branch-reduction}. After completing the
square and translating $x$ over $k$, write
$E:y^2=F(x)$, where $F(x)=x^3+Ax^2+Bx+C$. Replacing the
$k$-rational $x$-coordinate by
$x-\frac12(\theta+\theta^q)$ centers the two kernel roots, so we may
assume $\theta=x(\xi)$ satisfies $\theta^q=-\theta$. Then
$r=\theta^2\in k^\times$ is nonsquare. This uses only division by $2$. If
\(\operatorname{char}k=5\), the nonzero point
\(\xi\in E[5](\bar{k})\) forces \(E\) to be ordinary
\cite[Corollary~III.6.4(c)]{Silverman2009ArithmeticEllipticCurves};
thus the \emph{supersingular} case is absent, and
\eqref{eq:kernel-polynomial-isogeny} involves no division by \(5\).

Since \(x(2\xi)=-\theta\), the duplication formula
\cite[Group Law Algorithm~III.2.3]{Silverman2009ArithmeticEllipticCurves}
gives $F'(\theta)^2=4F(\theta)(A+\theta)$. Reduction modulo \(\theta^2-r\) yields
\begin{equation}
 C=Ar,\qquad B^2+2Br+5r^2=4A^2r.
\label{eq:rational-odd-duplication}
\end{equation}

Here \(Q_N=X^2-r\),
\[
 \frac{Q_N'}{Q_N}=\frac{2X}{X^2-r},\qquad
 \left(\frac{Q_N'}{Q_N}\right)'
 =-\frac{2(X^2+r)}{(X^2-r)^2},
\]
and \(b_2=4A,\ b_4=2B,\ b_6=4C\). Substitution in
\eqref{eq:kernel-polynomial-isogeny} gives, with \(Q=X^2-r\),
\[
 g(X)=5X+\frac{8F(X)(X^2+r)}{Q^2}
       -\frac{4XF'(X)}{Q}.
\]
Using \(C=Ar\) and collecting terms gives
\[
 g(X)=X+
 \frac{(4B+12r)X^3+24ArX^2+(12Br+4r^2)X+8Ar^2}
 {(X^2-r)^2}.
\]
Writing the four numerator coefficients as $a,b,c,d$ gives
\begin{equation}
 a=4B+12r,\qquad b=24Ar,\qquad
 c=12Br+4r^2,\qquad d=8Ar^2.
\label{eq:rational-odd-coefficients}
\end{equation}
Thus \(b=3d/r\) and \(c=3ar-32r^2\); moreover
\[
 d^2=64A^2r^4
 =16r^3(B^2+2Br+5r^2)
 =a^2r^3-16ar^4+128r^5.
\]
Substitution in the cubic discriminant gives
\begin{equation}
 \operatorname{disc}(F)=r^2(a-52r).
\label{eq:rational-odd-discriminant}
\end{equation}
Thus \(a\ne52r\). Finally, if
\(W=aX^3+bX^2+cX+d\), then $F(\theta)\ne0$: otherwise
$\xi=(\theta,0)$ would be a nonzero point of both orders $2$ and $5$.
Hence \(W(\theta)=16rF(\theta)\ne0\), so \(X^2-r\nmid W\).

Conversely, let $(r,a,d)$ satisfy the stated conditions and put
\[
 W(X)=aX^3+(3d/r)X^2+(3ar-32r^2)X+d.
\]
Modulo $X^2-r$, its remainder is
$4r(a-8r)X+4d$. Thus $X^2-r\mid W$ would give $a=8r$ and $d=0$,
whereas the defining equation would give $d^2=64r^5\ne0$. Hence
$X^2-r\nmid W$.

Put
\[
 A=\frac{d}{8r^2},\qquad B=\frac{a-12r}{4},\qquad
 C=Ar,\qquad F(X)=X^3+AX^2+BX+C.
\]
The parameter identities imply \eqref{eq:rational-odd-duplication} and
\eqref{eq:rational-odd-coefficients}, while
\eqref{eq:rational-odd-discriminant} is nonzero. Thus
\(E:y^2=F(x)\) is an elliptic curve
\cite[Proposition~III.1.4(a)]{Silverman2009ArithmeticEllipticCurves}.
Let \(\theta^2=r\). Then \(\theta^q=-\theta\). Since
\(W(\theta)=16rF(\theta)\) and $X^2-r\nmid W$, one has
\(F(\theta)\ne0\); choose \(\xi=(\theta,v)\) with
\(v^2=F(\theta)\). The duplication formula gives $x(2\xi)=-\theta=\theta^q$.
Lemma~\ref{lem:order-five-test} shows that $\xi$ has exact order $5$
and that $N=\langle\xi\rangle$ is Frobenius-stable. Hence $N$ defines
a separable quotient isogeny over $k$; in characteristic $5$, its
existence also forces $E$ to be ordinary.
Formula~\eqref{eq:kernel-polynomial-isogeny} gives $\mathcal F_{r,a,d}$.
Theorem~\ref{thm:intro-dihedral}\textup{(1)} gives reflection inertia,
so no branch point is totally ramified, while
Theorem~\ref{thm:intro-dihedral}\textup{(3)} gives exceptionality. Since
\(X^2-r\nmid W\), the fibre over \(\infty\) has partition
\(\{\!\{2,2,1\}\!\}\); thus \(\infty\) is a \(k\)-rational reflection
branch point.
\end{proof}

\begin{rem}[The boundary $a=52r$]
\label{rem:family-F-boundary}
Still in odd characteristic, the restriction $a\ne52r$ separates family
\textup{(F)} from the
totally ramified dihedral families; it is not an exceptionality
condition. Let $g$ denote the same Case~III expression with $a=52r$.
If $p\ne5$, then $d^2=2000r^5$. With
$\alpha:=-d/(20r^2)$ and $\beta:=5\alpha$, one has $\alpha^2=5r$;
direct substitution gives
\[
 D_5\!\left(X,\frac14\right)
 =\frac{25\alpha}{16\bigl(g(\alpha X/(X-1))-\beta\bigr)}+\frac1{16}.
\]
Thus the boundary belongs to family \textup{(D)}. Its existence forces
$5r$ to be a square; since $r$ is nonsquare, this is equivalent to
$\chi_2(5)=-1$, or $q\equiv\pm2\pmod5$.
If $p=5$, then $d=0$ and
\[
 g(X)=\frac{X^5}{(X^2-r)^2},\qquad
 \frac1{g(1/X)}=r^2X(X^2-r^{-1})^2,
\]
so the boundary belongs to family \textup{(E)}. In either case it
remains exceptional and has a totally ramified branch point.
\end{rem}

\begin{thm}
\label{rational-reflection-char2}
Let \(k=\mathbb{F}_q\) have characteristic \(2\). Let \(f\in k(X)\)
be separable and exceptional of degree \(5\), with \(G_f\cong D_{5}\).
Suppose that \(f\) has no totally ramified branch point and has a
\(k\)-rational reflection branch point. Then
\[
 f\sim_{k}X+\frac{W(X)}{(X^2+X+t)^2},\qquad
 \operatorname{Tr}_{k/\mathbb{F}_2}(t)=1,
\]
where either $W=1$, or $W=t^{-1}X^3+X$ with $t\ne1$.

Conversely, every such function is exceptional, has geometric
monodromy group isomorphic to \(D_{5}\), has no totally ramified branch point, and
has \(\infty\) as a \(k\)-rational reflection branch point.
\end{thm}

\begin{proof}
Use Proposition~\ref{rational-branch-reduction}, write
$E:y^2+a_1xy+a_3y=x^3+a_2x^2+a_4x+a_6$, and put
\(\theta=x(\xi)\). Set \(s=\theta+\theta^q\in k^\times\), and choose \(u\in k^\times\)
with \(u^2=s\). The admissible change \(x=u^2X,\ y=u^3Y\) makes the
two kernel coordinates differ by \(1\). Relabelling, we may assume
\[
 \theta^q=\theta+1,\qquad
 Q_t(\theta)=0,\qquad Q_t(X)=X^2+X+t.
\]
Thus $Q_t$ is irreducible; equivalently,
$\operatorname{Tr}_{k/\mathbb F_2}(t)=1$ by
\cite[Corollary~3.79]{LidlNiederreiter1997FiniteFields}.

In characteristic \(2\),
\[
 b_2=a_1^2,\qquad b_4=a_1a_3,\qquad b_6=a_3^2,
 \qquad Q_t'=1,\qquad
 \left(\frac{Q_t'}{Q_t}\right)'=\frac{1}{Q_t^2}.
\]
Formula~\eqref{eq:kernel-polynomial-isogeny} therefore gives
\begin{equation}
\begin{aligned}
 g(X)
 &=X+\frac{a_1^2X^2+a_3^2+a_1(a_1X+a_3)Q_t(X)}
 {Q_t(X)^2}\\
 &=X+\frac{(a_1X+a_3)(a_1X^2+a_1t+a_3)}
 {Q_t(X)^2}.
\end{aligned}
\label{eq:rational-char2-kernel}
\end{equation}

Suppose first that \(a_1=0\). Nonsingularity gives \(a_3\ne0\),
and duplication gives
\[
 \theta+1=x(2\xi)
 =\frac{\theta^4+a_4^2}{a_3^2}+a_2.
\]
As \(\theta^4=\theta+t+t^2\), comparison of the coefficients of
\(\theta\) gives \(a_3=1\). Hence
\[
 g(X)=X+\frac{1}{Q_t(X)^2}.
\]
Here \(E[2](\bar{k})=\{O_{E}\}\), since
\(-P=(x(P),y(P)+a_3)\); thus this is the supersingular case by the
criterion in
\cite[Corollary~III.6.4(c)]{Silverman2009ArithmeticEllipticCurves}.

Suppose \(a_1\ne0\), and put \(\mu=a_3/a_1\). Translating \(x\) by
\(\mu\) kills \(a_3\) and replaces \(t\) by \(t+\mu^2+\mu\), whose
absolute trace is unchanged; a
subsequent \(k\)-translation of \(y\) kills \(a_4\). We may write
\(E:y^2+a_1xy=x^3+a_2x^2+B\). Its discriminant is \(a_1^6B\), so
\(B\ne0\), and duplication gives
\(\theta+1=\theta^2/a_1^2+B/\theta^2\). Multiplying by \(\theta^2\)
and reducing modulo \(\theta^2+\theta+t\) yields
\((a_1^{-2}+t)\theta+a_1^{-2}(t+t^2)+B=0\). Hence
\(a_1^2=t^{-1}\), \(B=t^2(t+1)\), and \(t\ne1\); thus
\[
 g(X)=X+\frac{t^{-1}X^3+X}{Q_t(X)^2}.
\]
The point above \(x=0\) is a nonzero \(2\)-torsion point; hence this is
the ordinary case by
\cite[Corollary~III.6.4(c)]{Silverman2009ArithmeticEllipticCurves}.

Conversely, let \(Q_t\) be irreducible with root \(\theta\). For
\(W=1\), take \(E_0:y^2+y=x^3+(t^2+t+1)x^2\). For
\(W=t^{-1}X^3+X\), choose \(\gamma^2=t^{-1}\) and take
\(E_1:y^2+\gamma xy=x^3+t^2(t+1)\). Their discriminants are
\(1\) and \(\gamma^6t^2(t+1)\ne0\). Choose \(\xi\) above \(\theta\).
In either case,
\[
 x(2\xi)=\theta+1=\theta^q.
\]
Lemma~\ref{lem:order-five-test} gives a Frobenius-stable cyclic subgroup
$N=\langle\xi\rangle$ of order $5$. Hence $(E,N,O_E)$ is admissible, and
Formula~\eqref{eq:kernel-polynomial-isogeny} gives the displayed function.
Theorem~\ref{thm:intro-dihedral}\textup{(1)} gives $G_g\cong D_5$ with
reflection inertia, while part~\textup{(3)} gives exceptionality. For
$W=1$ there is no cancellation. For the ordinary form,
\(t^{-1}X^3+X\equiv t^{-1}X+1\pmod{Q_t}\), a nonzero linear
polynomial; hence $Q_t\nmid W$. Thus the two roots of
$Q_t$ are double poles and $\infty$ is the unique unramified point over
$\infty$. Therefore
$\operatorname{Ram}_\infty(g)=\{\!\{2,2,1\}\!\}$, so there is no
total ramification.
\end{proof}

\begin{rem}[Branch loci in Theorem~\ref{rational-reflection-char2}]
The two rows in Theorem~\ref{rational-reflection-char2} have different
geometric branch loci. If \(W=1\), then \(g'=1\), so
\(\operatorname{Br}(g)=\{\infty\}\). In the ordinary row, choose
\(\gamma^2=t^{-1}\) and put \(R=Q_t+\gamma X+1\). Then
\[
 g=\frac{XR^2}{Q_t^2},\qquad g'=\frac{R^2}{Q_t^2}.
\]
Since \(t\ne1\), the polynomials \(R\) and \(Q_t\) are separable and
coprime, and \(R(0)\ne0\). Hence
\(\operatorname{Br}(g)=\{0,\infty\}\).
\end{rem}

\section{Dihedral maps without a rational branch point: characteristic \texorpdfstring{\(2\)}{2}}
\label{sec:char2-no-rational}

We now treat the no-rational-branch case in characteristic \(2\). Computing
the shifted-Kummer trace yields family \(\mathrm{(H)}\). Let
$k=\mathbb{F}_{q}$ have
characteristic $2$, and let $f\in k(X)$ be a separable exceptional function of
degree $5$ such that $G_{f}\cong D_{5}$ and $\operatorname{Br}(f)\cap\mathbf{P}^{1}(k)=\varnothing$.
Corollary~\ref{total-branch-polynomial} excludes total ramification.
Proposition~\ref{dihedral-ramification} gives one or two reflection
branch points; the first possibility is excluded because a unique
branch point is Frobenius-fixed and hence $k$-rational. Thus the
branch locus consists of two reflection branch points.

By Theorem~\ref{thm:intro-dihedral}\textup{(1)}, $f\sim_{k}f_{E,N,\varepsilon}$
for an admissible isogeny datum $(E,N,\varepsilon)$ of degree $5$
over $k$. By Theorem~\ref{thm:intro-dihedral}\textup{(2)} and Theorem~\ref{thm:intro-dihedral}\textup{(3)},
\[
\varepsilon\notin2E(k),\qquad\sigma_q(P)=[\lambda]P\ (P\in N),\qquad\lambda\in\{\pm2\}\subset\mathbb F_5^\times.
\]
Lemma~\ref{lem:order-five-test} gives $N\cap E(k)=\{O_E\}$; since
$\varepsilon\notin2E(k)$, one has $\varepsilon\ne O_E$ and hence
$\varepsilon\notin N$. By Lemma~\ref{lem:defect-isomorphism}, the class
of $\varepsilon$ gives a nonzero element of $\mathfrak D(E)$, so
$E[2](k)\ne\{O_E\}$. In characteristic $2$, this forces $E$ to be
ordinary and $E[2](k)$ to have order $2$; thus $E(k)/2E(k)$ has a
unique nonzero class. By an admissible Weierstrass change \cite[Proposition~A.1.1(c)]{Silverman2009ArithmeticEllipticCurves},
we may put
\[
E:y^{2}+xy=x^{3}+Ax^{2}+\beta,\qquad\beta\ne0.
\]
The inverse is $-(x,y)=(x,y+x)$; hence the unique nonzero point of
$E[2]$ is $(0,\sqrt{\beta})$, which is also the unique affine $k$-point
with $x=0$.

This nonzero class has a representative with nonzero $x$-coordinate.
Otherwise it would consist solely of the point $(0,\sqrt{\beta})$.
Since all cosets of $2E(k)$ have the same cardinality, this would
give $2E(k)=\{O_{E}\}$ and $|E(k)|=2$. The Hasse bound \cite[Theorem~V.1.1]{Silverman2009ArithmeticEllipticCurves}
then gives $q=2$ or $q=4$. As $5\ne\operatorname{char}k$, one has
$E[5]\cong(\mathbb{Z}/5\mathbb{Z})^{2}$ by \cite[Corollary~III.6.4(b)]{Silverman2009ArithmeticEllipticCurves}.
With $a_{q}=q+1-|E(k)|=q-1$, the $q$-power Frobenius endomorphism is
annihilated by $T^{2}-a_{q}T+q$ by
\cite[Theorem~V.2.3.1(b)]{Silverman2009ArithmeticEllipticCurves}; hence
the same polynomial annihilates its action on $E[5]$. The stable
line $N$ would therefore give a root of this polynomial in $\mathbb{F}_{5}$,
but for both $q=2$ and $q=4$ its discriminant is $-7\equiv3\pmod 5$,
a nonsquare.

By Theorem~\ref{thm:intro-dihedral}\textup{(2)}, the $k$-M\"obius class of
$f_{E,N,\varepsilon}$ is unchanged when $\varepsilon$ is replaced
modulo $2E(k)$. Since translations commute, this conjugation leaves
$\{\tau_{\alpha}:\alpha\in N\}$ unchanged. We may therefore write
$\varepsilon=(u,v)$ with $u\ne0$. The admissible change $y\mapsto y+(v/u)x$
sends $\varepsilon$ to $(u,0)$; it fixes all $x$-coordinates and
the constant term $\beta$, and changes only $A$. Thus we may write
\[
E:y^{2}+xy=x^{3}+(u+\eta)x^{2}+\beta,\qquad\varepsilon=(u,0),\qquad\eta=\beta/u^{2}.
\]
This marked model is rigid. By
Lemma~\ref{lem:marked-Weierstrass}\textup{(2)}, a residual isomorphism
has $u_0=1$ and $\rho=0$. Preservation of $a_4=0$ then gives
$\tau=0$, while preservation of $\varepsilon=(u,0)$ gives $\nu u=0$.
Since $u\ne0$, also $\nu=0$.

Choose a generator $\xi=(t,\theta)$ of $N$. Since Frobenius acts
by $\pm2$, one has $t^{q}=x(2\xi)\ne t$ and $t^{q^{2}}=t$. Write
$(X-t)(X-t^{q})=X^{2}+rX+w$. In characteristic $2$, $t^{q}=t+r$
and $tt^{q}=w$. The group-law formulae in \cite[Group Law Algorithm~III.2.3]{Silverman2009ArithmeticEllipticCurves}
give
\[
x(2P)=x(P)^{2}+\frac{\beta}{x(P)^{2}}.
\]
Applying it to $\xi$ gives $t^{4}+\beta=(t+r)t^{2}$. Reducing this
identity modulo $t^{2}+rt+w$ gives $(w+r^{3})t+\beta+r^{2}w+w^{2}=0$.
Since $t\notin k$, the two coefficients vanish: $w=r^{3}$ and $\beta=r^{5}(r+1)$.
Consequently
\[
\eta=\frac{r^{5}(r+1)}{u^{2}},\qquad t^{2}+rt+r^{3}=0.
\]
Replacing $\xi$ by another generator of $N$ only exchanges $t$
and $t^{q}$; hence $r$ and $u$ are unchanged. The polynomial $X^{2}+rX+r^{3}$
is irreducible over $k$. Thus $r\ne0$, and after putting $X=rY$
its irreducibility is equivalent to that of $Y^{2}+Y+r$. By the Artin--Schreier
criterion \cite[Corollary~3.79]{LidlNiederreiter1997FiniteFields},
this is equivalent to $\operatorname{Tr}_{k/\mathbb{F}_{2}}(r)=1$.
Since $\beta\ne0$, also $r\ne1$.
\begin{prop}
\label{char2-order-five} Let $r,u\in k$ satisfy $r\ne1$, $u\ne0$,
and put $\eta=r^{5}(r+1)/u^{2}$. Assume $X^{2}+rX+r^{3}$ is irreducible
over $k$. Let
\[
E:y^{2}+xy=x^{3}+(u+\eta)x^{2}+r^{5}(r+1),
\]
let $t$ be a root of $X^{2}+rX+r^{3}$, and let $\xi\in E(\bar{k})$
satisfy $x(\xi)=t$. Then $5\xi=O_{E}$ and $\sigma_{q}(\xi)\in\{\pm2\xi\}$.
\end{prop}

\begin{proof}
Irreducibility gives $r\ne0$, and $r\ne1$ gives
$\Delta_E=r^5(r+1)\ne0$; hence $E$ is elliptic. From $t^{2}+rt+r^{3}=0$
one gets $t^{4}+r^{5}(r+1)=(t+r)t^{2}$. The duplication formula gives $x(2\xi)=t+r=t^q$. The assertion
now follows from Lemma~\ref{lem:order-five-test}.
\end{proof}

\begin{prop}
\label{char2-source-coordinate} Define $z:=(y+u)/(x+u)\in k(E)$.
Then $z\circ\iota_{\varepsilon}=z$ and $k(z)=k(E)^{\langle\iota_{\varepsilon}\rangle}$.
Moreover
\[
x^{2}+(z^{2}+z+\eta)x+u(z^{2}+1+\eta)=0.
\]
For $P\in E(\bar{k})\setminus\{O_{E},\varepsilon\}$, $\iota_{\varepsilon}(P)=P$
if and only if $z(P)^{2}+z(P)+\eta=0$. Consequently
\[
\varepsilon\notin2E(k)\quad\Longleftrightarrow\quad X^{2}+X+\eta\text{ has no root in }k\quad\Longleftrightarrow\quad\operatorname{Tr}_{k/\mathbb{F}_{2}}(\eta)=1.
\]
\end{prop}

\begin{proof}
For $P\notin\{\pm\varepsilon\}$, the line $y+u=z(P)(x+u)$ passes
through $P$ and $-\varepsilon=(u,u)$. By the chord-and-tangent law
in \cite[Group Law Algorithm~III.2.3]{Silverman2009ArithmeticEllipticCurves},
its third intersection with $E$ is $\varepsilon-P=\iota_{\varepsilon}(P)$.
Hence $z$ is fixed by $\iota_{\varepsilon}$.

Substitute $y=z(x+u)+u$ in the equation of $E$ and use $\beta=u^{2}\eta$.
One obtains $(x+u)(x^{2}+(z^{2}+z+\eta)x+u(z^{2}+1+\eta))=0$. Since
$x+u$ is not the zero function, the displayed quadratic holds in
$k(E)$. Thus $k(E)=k(x,z)$ is at most quadratic over $k(z)$. It
is not equal to $k(z)$, since $E$ has genus one. Hence $[k(E):k(z)]=2$,
which is also the degree of the quotient by the nontrivial involution
$\iota_{\varepsilon}$; therefore $k(z)=k(E)^{\langle\iota_{\varepsilon}\rangle}$.
The orders $\operatorname{ord}_{O_{E}}(x)=-2$ and $\operatorname{ord}_{O_{E}}(y)=-3$
are computed in \cite[Proposition~III.3.1(a)]{Silverman2009ArithmeticEllipticCurves},
so $z$ has a simple pole at $O_{E}$. Since $z$ is invariant and
$\iota_{\varepsilon}(O_{E})=\varepsilon$, it also has a simple pole
at $\varepsilon$; as $\deg z=2$, these are its only poles. In
particular, the apparent $0/0$ in $(y+u)/(x+u)$ at
$-\varepsilon=(u,u)$ is removable; the tangent relation there gives the
regular value $z(-\varepsilon)=u+1$.

The two roots of the quadratic in $x$ are exchanged by $\iota_{\varepsilon}$.
They coincide exactly when the coefficient of $x$ is zero, that is, when
$z^{2}+z+\eta=0$. The poles $O_{E}$ and $\varepsilon$ are exchanged
by $\iota_{\varepsilon}$, so no fixed point lies above $z=\infty$.
Since $\iota_\varepsilon(P)=P$ is equivalent to
$2P=\varepsilon$, let $\zeta\in k$ satisfy
$\zeta^2+\zeta+\eta=0$. The displayed equation is then a square in $x$ and, since squaring is
bijective on the finite field $k$, has a unique root $x_0\in k$. If $x_0\ne u$, it gives the
required fixed point. If $x_0=u$, then $u\ne0$ and the two equations
give $\zeta=u+1$ and $\eta=u^2+u$. Since $\beta=u^2\eta$, the
duplication formula yields
$ x(2\varepsilon)=u^2+\beta/u^2=u^2+\eta=u=x(\varepsilon)$.
As $\varepsilon\ne O_E$, one has $2\varepsilon=-\varepsilon$;
hence $2(-\varepsilon)=\varepsilon$. Thus a $k$-rational half exists
also in this degenerate case. Conversely, if $P\in E(k)$ and $2P=\varepsilon$,
then $P\notin\{O_{E},\varepsilon\}$, since either equality would
give $\varepsilon=O_{E}$; hence $z(P)\in k$ is defined and is a
root. The trace criterion is \cite[Corollary~3.79]{LidlNiederreiter1997FiniteFields}.
\end{proof}

For $T\in E(\bar{k})$, define \(H_T(P):=z(P+T)+z(P-T)\).
Since $\varepsilon\notin N$, Theorem~\ref{thm:shifted-trace-coordinate}
provides the target coordinate. Write $\Psi$ for its pullback to $E$;
then
\[
 \Psi(P)=\sum_{\alpha\in N}z(P+\alpha),
 \qquad \Psi=z+H_\xi+H_{2\xi}.
\]

For a finite field extension $L/K$, let $\operatorname{N}_{L/K}$
denote the \emph{norm}.
\begin{prop}
\label{char2-Hxi} Assume the hypotheses of Proposition~\ref{char2-order-five}
and retain the notation above. Then
\[
H_{\xi}=\frac{t(z^{2}+z+\eta)}{L_t(z)}\in k(t)(z),
\]
where $L_t(X):=(u+t)X^2+tX+(r+\eta)t+u(1+\eta)+r^3$.
\end{prop}

\begin{proof}
Since $\varepsilon\in E(k)$, Frobenius fixes $\varepsilon$; since
it acts on $N$ by $\pm2$ and $u\ne0$, one has $\varepsilon\notin N$.
Put $K=k(t,\theta)$ and $h(P)=z(P+\xi)$. By Proposition~\ref{char2-source-coordinate},
$K(E)^{\langle\iota_{\varepsilon}\rangle}=K(z)$, and $h(\iota_{\varepsilon}(P))=z(P-\xi)$.
Hence
\[
H_{\xi}=h+h\circ\iota_{\varepsilon}=\operatorname{Tr}_{K(E)/K(z)}(h)\in K(z).
\]
The function $z$ has simple poles precisely at $O_{E}$ and $\varepsilon$.
Since $\varepsilon\notin N$, the four points $\pm\xi,\varepsilon\pm\xi$
are distinct. Thus $H_{\xi}$, viewed as a function of $z$, has at
most the two simple poles $z(\xi)$ and $z(-\xi)$.

The conjugate of $x$ over $K(z)$ is $x+z^{2}+z+\eta$. Hence
\[
\operatorname{N}_{K(E)/K(z)}(x+t)=(x+t)(x+z^{2}+z+\eta+t)=L_{t}(z),
\]
where $t^{2}=rt+r^{3}$ is used. Since $x+t$ has simple zeros at
$\xi$ and $-\xi$, $L_{t}(X)=(u+t)(X+z(\xi))(X+z(-\xi))$. It follows
that
\[
H_{\xi}=\frac{M(z)}{L_{t}(z)}\quad\text{with}\quad M\in K[X],\ \deg M\leqslant2,
\]
the degree bound following also from the finiteness of $H_{\xi}$
at $z=\infty$.

If $P$ is fixed by $\iota_{\varepsilon}$, then $z(P)^{2}+z(P)+\eta=0$
and $\iota_{\varepsilon}(P-\xi)=P+\xi$. Thus $z(P-\xi)=z(P+\xi)$,
and $H_{\xi}(P)=0$ in characteristic $2$. Neither of the two roots
of $X^{2}+X+\eta$ is a pole: otherwise $2(\pm\xi)=\varepsilon\in N$.
Since these roots are distinct, $M(X)=c(X^{2}+X+\eta)$ for some $c\in K$.

At $O_{E}$, where $z=\infty$ and $-\xi=(t,\theta+t)$,
\[
H_{\xi}(\infty)=z(\xi)+z(-\xi)=\frac{\theta+u}{t+u}+\frac{\theta+t+u}{t+u}=\frac{t}{t+u}.
\]
Since $(z^{2}+z+\eta)/L_{t}(z)$ tends to $1/(u+t)$, one gets $c=t\ne0$.
\end{proof}

Write $L_t=U_{r,u}+tV_{r,u}$, with
\[
U_{r,u}(X)=uX^2+u(1+\eta)+r^3,\qquad
V_{r,u}(X)=X^2+X+r+\eta.
\]
\begin{prop}
\label{char2-formula} Assume the hypotheses of Propositions~\ref{char2-order-five}
and~\ref{char2-Hxi} and retain the notation above. The target coordinate whose pullback is $\Psi$
is a $k$-coordinate on $(E/N)/\langle\iota_{\varphi(\varepsilon)}\rangle$.
In this coordinate, $\Psi=\mathcal{H}_{r,u}(z)$, where
\[
\mathcal{H}_{r,u}(X)=X+\frac{r(X^{2}+X+\eta)U_{r,u}(X)}{U_{r,u}(X)^{2}+rU_{r,u}(X)V_{r,u}(X)+r^{3}V_{r,u}(X)^{2}}.
\]
The denominator is irreducible of degree $4$, no cancellation occurs, and $\deg \mathcal H_{r,u}=5$.
\end{prop}

\begin{proof}
The point $2\xi$ has $x$-coordinate $t^{q}$. Applying Proposition~\ref{char2-Hxi} to $2\xi$ replaces $t$ by $t^{q}$. Since $t+t^{q}=r$
and $tt^{q}=r^{3}$,
\[
L_{t}L_{t^{q}}=U^{2}_{r,u}+rU_{r,u}V_{r,u}+r^{3}V^{2}_{r,u},\qquad tL_{t^{q}}+t^{q}L_{t}=rU_{r,u}.
\]
Together with $\Psi=z+H_{\xi}+H_{2\xi}$, these identities give the
displayed formula.

The leading coefficient of the denominator is $u^{2}+ru+r^{3}=(u+t)(u+t^{q})\ne0$,
since $u\in k$ whereas $t,t^{q}\notin k$. Its roots are $z(\xi),z(-\xi),z(2\xi),z(-2\xi)$.
They are distinct: if $z(\alpha)=z(\beta)$ with $\alpha,\beta\in N$,
then $\beta=\alpha$ or $\beta=\varepsilon-\alpha$; the latter would
give $\varepsilon\in N$. Frobenius acts on $N\setminus\{O_{E}\}$
by $2$ or $-2$, hence as a four-cycle on these roots. The denominator
is therefore irreducible of degree $4$.

At a root of $L_{t}$, the numerator of $H_{\xi}$ is nonzero, for
otherwise Proposition~\ref{char2-source-coordinate} would give $2(\pm\xi)=\varepsilon\in N$;
the preceding distinctness also gives $L_{t^{q}}\ne0$ there, so $H_{2\xi}$
has no pole. The conjugate argument applies to $L_{t^{q}}$. Thus no cancellation occurs. The four simple finite poles and the
simple pole at infinity give $\deg \mathcal H_{r,u}=5$, while
Theorem~\ref{thm:shifted-trace-coordinate} identifies the displayed
function with the lower morphism in the chosen quotient coordinates.
\end{proof}

\begin{thm}
\label{char2-no-rational} Let $k=\mathbb{F}_{q}$ have characteristic
$2$, and let $f\in k(X)$ be separable of degree $5$. Then the following
are equivalent.
\begin{enumerate}
\item[(1)] $f$ is exceptional, $G_{f}\cong D_{5}$, and $\operatorname{Br}(f)\cap\mathbf{P}^{1}(k)=\varnothing$.
\item[(2)] $f\sim_{k}\mathcal{H}_{r,u}$ for some $r,u\in k$ satisfying
\[
r\ne1,\qquad u\ne0,\qquad\operatorname{Tr}_{k/\mathbb{F}_{2}}(r)=1,\qquad\operatorname{Tr}_{k/\mathbb{F}_{2}}\!\left(\frac{r^{5}(r+1)}{u^{2}}\right)=1.
\]
\end{enumerate}
In this case the branch locus consists of two reflection branch points.
\end{thm}

\begin{proof}
Assume (1). The preceding construction gives $r\ne1$, $u\ne0$,
$\operatorname{Tr}_{k/\mathbb F_2}(r)=1$, and
$\eta=r^5(r+1)/u^2$. Since the branch locus has no $k$-point,
Theorem~\ref{thm:intro-dihedral}\textup{(2)} gives
$\varepsilon\notin2E(k)$. Proposition~\ref{char2-source-coordinate}
and the Artin--Schreier criterion
\cite[Corollary~3.79]{LidlNiederreiter1997FiniteFields} then give
$\operatorname{Tr}_{k/\mathbb F_2}(\eta)=1$, while
Proposition~\ref{char2-formula} gives
$f\sim_k\mathcal H_{r,u}$.

Conversely, the two trace conditions make
$X^2+rX+r^3$ and $X^2+X+\eta$ irreducible. Propositions~\ref{char2-order-five} and~\ref{char2-source-coordinate} therefore
construct an admissible datum $(E,N,\varepsilon)$ with multiplier
$\pm2$ and $\varepsilon\notin2E(k)$; by
Lemma~\ref{lem:order-five-test}, also $\varepsilon\notin N$.
Proposition~\ref{char2-formula} identifies its lower morphism with
$\mathcal H_{r,u}$, and Theorem~\ref{thm:intro-dihedral} gives
exceptionality, monodromy $D_5$, and no rational branch point. Total
ramification is impossible by Corollary~\ref{total-branch-polynomial};
Proposition~\ref{dihedral-ramification} then gives exactly two
reflection branch points.
\end{proof}

\section{Dihedral maps without a rational branch point: odd characteristic}
\label{sec:odd-no-rational}

We next treat the no-rational-branch case in odd characteristic. The
translation trace yields family \(\mathrm{(I)}\). Let
$k=\mathbb{F}_{q}$ have odd characteristic, and let $f\in k(X)$ be a
separable exceptional function of degree $5$ such that
$G_f\cong D_5$ and
$\operatorname{Br}(f)\cap\mathbf P^1(k)=\varnothing$.
Corollary~\ref{total-branch-polynomial} excludes total ramification,
and Proposition~\ref{dihedral-ramification}(1), or (3) in characteristic
$5$, gives four tame reflection branch points. By
Theorem~\ref{thm:intro-dihedral}, $f\sim_k f_{E,N,\varepsilon}$ for an
admissible datum of degree $5$ such that
\[
 \varepsilon\notin2E(k),\qquad
 \sigma_q(P)=[\lambda]P\quad(P\in N),\qquad
 \lambda\in\{\pm2\}.
\]
Lemma~\ref{lem:order-five-test} gives $N\cap E(k)=\{O_E\}$; hence
$\varepsilon\notin N$.

Since $\varepsilon\ne O_E$, a $k$-rational admissible change sends it
to $(0,0)$. Since $2\ne0$ in $k$, the $xy$-term can be removed; see
\cite[Proposition~III.3.1(b)]{Silverman2009ArithmeticEllipticCurves}.
Thus we may write
\[
\begin{gathered}
 E:y^2+cy=x^3+ax^2+bx,\qquad \varepsilon=(0,0),\\
 \Delta_E=-16a^3c^2+16a^2b^2+72abc^2-64b^3-27c^4\ne0.
\end{gathered}
\]
Here $-(x,y)=(x,-c-y)$ and $-\varepsilon=(0,-c)$.

Choose a generator $\xi=(t,\theta)$ of $N$. Since Frobenius acts by
$\pm2$, one has $t^q=x(2\xi)$ and $t^{q^2}=t$. The element $t$ does not
lie in $k$, for otherwise $x(2\xi)=x(\xi)$ would give $2\xi=\pm\xi$.
Write
\[
 (X-t)(X-t^q)=X^2+rX+s,
\]
so $t^q=-r-t$, $tt^q=s$, and this quadratic is irreducible. The
duplication formula is
\[
 x(2P)=\left(\frac{3x(P)^2+2ax(P)+b}{2y(P)+c}\right)^2-a-2x(P).
\]
Together with the identity $(2y+c)^2=4(x^3+ax^2+bx)+c^2$, it gives
\[
 (3t^2+2at+b)^2=(a+t-r)\bigl(4(t^3+at^2+bt)+c^2\bigr).
\]
Reducing modulo $t^2+rt+s$ shows that this identity is equivalent to
\begin{equation}
 -4as+2br-c^2-r^3+6rs=0,\qquad
 -ac^2+b^2-2bs+c^2r-r^2s+5s^2=0.
\label{eq:odd-duplication-conditions}
\end{equation}

\begin{prop}
\label{odd-order-five}
Assume that $E:y^2+cy=x^3+ax^2+bx$ is nonsingular,
$X^2+rX+s$ is irreducible over $k$, and
\eqref{eq:odd-duplication-conditions} holds. If $t$ is a root of
$X^2+rX+s$ and $\xi\in E(\bar k)$ satisfies $x(\xi)=t$, then
$\xi$ has exact order $5$ and $\sigma_q(\xi)\in\{\pm2\xi\}$.
\end{prop}

\begin{proof}
The two equations say precisely that the preceding duplication identity
holds at $t$. If $\xi=(t,\theta)$ and $2\theta+c=0$, that identity
forces $3t^2+2at+b=0$, so both partial derivatives of the Weierstrass
equation vanish at $\xi$, contrary to nonsingularity. Thus duplication
is valid and gives $x(2\xi)=-r-t=t^q$. Lemma~\ref{lem:order-five-test}
completes the proof.
\end{proof}

\begin{prop}
\label{odd-source-coordinate}
Define $z:=(y+c)/x\in k(E)$. Then
$z\circ\iota_\varepsilon=z$ and
$k(z)=k(E)^{\langle\iota_\varepsilon\rangle}$. Moreover
\[
 x^2+(a-z^2)x+b+cz=0.
\]
For $P\in E(\bar k)\setminus\{O_E,\varepsilon\}$, one has
$\iota_\varepsilon(P)=P$ if and only if
$R_{a,b,c}(z(P))=0$, where
\[
 R_{a,b,c}(X):=X^4-2aX^2-4cX+a^2-4b.
\]
The discriminant of $R_{a,b,c}$ is $256\Delta_E$. Consequently,
$\varepsilon\notin2E(k)$ if and only if $R_{a,b,c}$ has no root in
$k$.
\end{prop}

\begin{proof}
For $P\notin\{\pm\varepsilon\}$, the value $z(P)$ is the slope of the line
through $P$ and $-\varepsilon=(0,-c)$; the chord law shows that the third
intersection is $\varepsilon-P$. Hence $z$ is fixed by $\iota_\varepsilon$.
Substitution of $y=zx-c$ gives the displayed quadratic. Its other root
is $z^2-a-x$, so $[k(E):k(z)]\leqslant2$. Equality cannot be $1$ because
$E$ has genus one. Hence $[k(E):k(z)]=2$, and
$k(z)=k(E)^{\langle\iota_\varepsilon\rangle}$. The orders of $x$ and $y$
at $O_E$
\cite[Proposition~III.3.1(a)]{Silverman2009ArithmeticEllipticCurves}
show that $z$ has simple poles precisely at $O_E$ and $\varepsilon$.
If $c\ne0$, the apparent $0/0$ in $(y+c)/x$ at
$-\varepsilon=(0,-c)$ is removable; the tangent relation there gives
$z(-\varepsilon)=-b/c$. If $c=0$, then
$-\varepsilon=\varepsilon$ is already one of the two poles.

For finite $z=Z$, the two points above $Z$ coincide precisely when the
quadratic in $x$ has discriminant
$(a-Z^2)^2-4(b+cZ)=R_{a,b,c}(Z)$. Suppose first that
$Z\in k$ is a root and put $x_0=(Z^2-a)/2$. If $x_0\ne0$, then
$(x_0,Zx_0-c)\in E(k)$ is fixed by $\iota_\varepsilon$, and hence is
a rational half of $\varepsilon$. If $x_0=0$, then $Z^2=a$ and
$b+cZ=0$. Nonsingularity gives $c\ne0$, while the duplication formula
at $\varepsilon$ gives
$x(2\varepsilon)=(b/c)^2-a=0=x(\varepsilon)$. Thus
$2\varepsilon=-\varepsilon$, so $-\varepsilon$ is a rational half of
$\varepsilon$.

Conversely, let $P\in E(k)$ satisfy $2P=\varepsilon$. If
$P\ne-\varepsilon$, then $P\notin\{O_E,\varepsilon\}$ and
$z(P)\in k$ is a root of $R_{a,b,c}$. If $P=-\varepsilon$, then
$2\varepsilon=-\varepsilon$; as above, $c\ne0$ and
$Z=-b/c\in k$ satisfies $Z^2=a$, hence $R_{a,b,c}(Z)=0$. This proves
the stated equivalence. Finally, computing the quartic discriminant gives
$\operatorname{disc}R_{a,b,c}=256\Delta_E\ne0$.
\end{proof}

For \(T\in E(\bar k)\), define \(H_T\in\bar k(E)\) by
\(H_T(P):=z(P+T)+z(P-T)\) for \(P\in E(\bar k)\), wherever the
right-hand side is defined.

\begin{prop}
\label{odd-Hxi}
Assume the hypotheses of Proposition~\ref{odd-order-five}, put
$N:=\langle\xi\rangle$, and assume $\varepsilon\notin N$. Define
\[
\begin{aligned}
 A(X)&=X^2+r-a, & B(X)&=s-b-cX,\\
 C(X)&=2(c+rX), & D(X)&=c(X^2+a)+2(b+s)X.
\end{aligned}
\]
\[
\hbox to \displaywidth{%
  \rlap{Then}\hfil
  $\displaystyle H_\xi=\frac{C(z)t+D(z)}{A(z)t+B(z)}\in k(t)(z)$.
  \hfil
}
\]
\end{prop}

\begin{proof}
Put $K=k(t,\theta)$. Since $z\circ\iota_\varepsilon=z$ and
$\iota_\varepsilon(P-\xi)=\varepsilon-P+\xi$, the two summands in
$H_\xi$ are conjugate over $K(z)$; hence $H_\xi\in K(z)$. The
conjugate of $x$ over $K(z)$ is $z^2-a-x$, and therefore
\[
 \operatorname{N}_{K(E)/K(z)}(x-t)
 =(x-t)(z^2-a-x-t)=-(A(z)t+B(z)).
\]
Write $P=(x,zx-c)$ and $\xi=(t,\theta)$. The chord law gives
\[
 z(P+\xi)=\frac{U_+}{(t-x)(zt-\theta-c)},\qquad
 z(P-\xi)=\frac{U_-}{(t-x)(zt+\theta)},
\]
\[
\hbox to \displaywidth{%
  \rlap{where}\hfil
  $\displaystyle
  \begin{aligned}
   U_+&=-(\theta-zx+c)(zt-\theta-c)-t(t-x)^2,\\
   U_-&=(\theta+zx)(zt+\theta)-t(t-x)^2.
  \end{aligned}$
  \hfil
}
\]
The relations
\[
 x^2=(z^2-a)x-b-cz,\qquad t^2=-rt-s,\qquad
 \theta^2=-c\theta+(r^2-s-ar+b)t+s(r-a)
\]
give, after substitution and collection of terms,
\[
\begin{aligned}
 (zt-\theta-c)(zt+\theta)&=t\bigl(A(z)t+B(z)\bigr),\\
 U_+(zt+\theta)+U_-(zt-\theta-c)
 &=t(t-x)\bigl(C(z)t+D(z)\bigr).
\end{aligned}
\]
Since irreducibility gives $s\ne0$ and hence $t\ne0$, adding the two
chord-law formulas and cancelling the common factors proves the result.
\end{proof}

Write $\Psi$ for the pullback to $E$ of the target coordinate in
Theorem~\ref{thm:shifted-trace-coordinate}; thus
\(\Psi(P)=\sum_{\alpha\in N}z(P+\alpha)\) for
\(P\in E(\bar k)\), wherever the right-hand side is defined.

\begin{prop}
\label{odd-formula}
Under the preceding hypotheses, the target coordinate whose pullback is
$\Psi$ is a $k$-coordinate on
$(E/N)/\langle\iota_{\varphi(\varepsilon)}\rangle$, and
$\Psi=\mathcal I_{a,b,c,r,s}(z)$, where
\[
\begin{aligned}
 \mathcal I_{a,b,c,r,s}(X)
 &=X\\
 &\quad+\frac{2sA(X)C(X)-r\bigl(A(X)D(X)+B(X)C(X)\bigr)+2B(X)D(X)}
 {B(X)^2-rA(X)B(X)+sA(X)^2}.
\end{aligned}
\]
The denominator is irreducible of degree $4$, no cancellation occurs,
and the resulting rational function has degree $5$.
\end{prop}

\begin{proof}
Theorem~\ref{thm:shifted-trace-coordinate} applies because
$\varepsilon\notin N$. Since $N=\{O_E,\pm\xi,\pm2\xi\}$,
\[
 \Psi=z+H_\xi+H_{2\xi}.
\]
Because $\sigma_q(\xi)=\pm2\xi$ and $H_{-T}=H_T$, the formula for
$H_{2\xi}$ is the Frobenius conjugate of that for $H_\xi$. Hence
\[
\begin{aligned}
 H_\xi+H_{2\xi}
 &=\operatorname{Tr}_{k(t)/k}
   \left(\frac{C(z)t+D(z)}{A(z)t+B(z)}\right)\\
 &=\frac{2sA(z)C(z)
 -r\bigl(A(z)D(z)+B(z)C(z)\bigr)+2B(z)D(z)}
 {B(z)^2-rA(z)B(z)+sA(z)^2},
\end{aligned}
\]
which gives the displayed expression.

The roots of $A(X)t+B(X)$ are $z(\xi)$ and $z(-\xi)$, and its
Frobenius conjugate has roots $z(2\xi)$ and $z(-2\xi)$. These four
values are distinct, for equality would imply $\varepsilon\in N$.
Frobenius acts on them as a four-cycle, and the leading coefficient of
the denominator is $s\ne0$; hence the denominator is irreducible of
degree $4$. Any cancellation would lower the degree below $5$, contradicting
Theorem~\ref{thm:shifted-trace-coordinate}.
\end{proof}

\medskip
\noindent\textbf{Odd-admissible quintuples.}
A quintuple $(a,b,c,r,s)\in k^5$ is \emph{odd-admissible} if
\begin{enumerate}
\item[(I1)] $-16a^3c^2+16a^2b^2+72abc^2-64b^3-27c^4\ne0$;
\item[(I2)] $X^2+rX+s$ is irreducible over $k$;
\item[(I3)] $-4as+2br-c^2-r^3+6rs=0$ and $-ac^2+b^2-2bs+c^2r-r^2s+5s^2=0$;
\item[(I4)] $R_{a,b,c}(X)$ has no root in $k$.
\end{enumerate}
For such a quintuple, define $A,B,C,D$ and
$\mathcal I_{a,b,c,r,s}$ as above.

This definition is invariant under the residual $k$-rational Weierstrass
changes. Let a marked $k$-isomorphism between two such models be written as
\[
 x=u_0^2x'+\rho,
 \qquad
 y=u_0^3y'+u_0^2\nu x'+\tau.
\]
Here $u_0\in k^\times$ and $\rho,\nu,\tau\in k$.
Both models have $a_1=0$. Since the characteristic is odd, the
transformation formula $u_0 a_1'=a_1+2\nu$
\cite[Table~3.1, p.~45]{Silverman2009ArithmeticEllipticCurves} gives
$\nu=0$.
Preservation of the marked point $(0,0)$ then gives $\rho=\tau=0$.
Thus only the scaling
\[
 (a,b,c,r,s)\longmapsto
 (u_0^{-2}a,u_0^{-4}b,u_0^{-3}c,
  u_0^{-2}r,u_0^{-4}s)
\]
remains. It preserves (I1)--(I4), and direct substitution gives
\[
 \mathcal I_{a',b',c',r',s'}(X)
 =u_0^{-1}\mathcal I_{a,b,c,r,s}(u_0 X).
\]

\begin{thm}
\label{odd-no-rational}
Let $k=\mathbb F_q$ have odd characteristic and let $f\in k(X)$ be
separable of degree $5$. The following are equivalent:
\begin{enumerate}
\item[(1)] $f$ is exceptional, $G_f\cong D_5$, and
$\operatorname{Br}(f)\cap\mathbf P^1(k)=\varnothing$;
\item[(2)] $f\sim_k\mathcal I_{a,b,c,r,s}$ for an odd-admissible
quintuple $(a,b,c,r,s)$.
\end{enumerate}
In this case the branch locus consists of four reflection branch points.
\end{thm}

\begin{proof}
Assume (1). The normalization above gives the marked curve
$E:y^2+cy=x^3+ax^2+bx$. Nonsingularity is (I1), the minimal
polynomial of $x(\xi)$ gives (I2), the duplication calculation gives
(I3), and Proposition~\ref{odd-source-coordinate} identifies
$\varepsilon\notin2E(k)$ with (I4). Proposition~\ref{odd-formula}
then gives $f\sim_k\mathcal I_{a,b,c,r,s}$.

Conversely, an odd-admissible quintuple defines the marked elliptic
curve $E$ and $\varepsilon=(0,0)$. Conditions (I2)--(I3) and
Proposition~\ref{odd-order-five} give a Frobenius-stable cyclic subgroup
$N=\langle\xi\rangle$ of order $5$ with multiplier $\pm2$;
Lemma~\ref{lem:order-five-test} gives $\varepsilon\notin N$. Condition
(I4) gives $\varepsilon\notin2E(k)$. Theorem~\ref{thm:intro-dihedral}
and Proposition~\ref{odd-formula} now give the asserted exceptional map,
with $G_f\cong D_5$ and no rational branch point. A totally ramified branch
point would be $k$-rational by Corollary~\ref{total-branch-polynomial}, so
none exists; Proposition~\ref{dihedral-ramification} then gives four
reflection branch points.
\end{proof}

\section{M\"obius equivalence and class counts}
\label{sec:parameter-equivalence}

We first assemble the preceding results to prove Theorem~\ref{main}; we
then determine the parameter identifications and class counts.

\begin{proof}[Proof of Theorem~\ref{main}.]
The inseparable case is \(\mathrm{(A)}\) with $p=5$. Assume that $f$ is
separable. Proposition~\ref{monodromy-reduction} gives $G_f\cong C_5$ or
$D_5$. For $C_5$, Theorem~\ref{cyclic-classification} and
Proposition~\ref{additive-char5} give \(\mathrm{(A)}\)--\(\mathrm{(C)}\).
For $D_5$, total ramification gives \(\mathrm{(D)},\mathrm{(E)}\); the
rational-branch theorems give \(\mathrm{(F)},\mathrm{(G)}\), and
Theorems~\ref{char2-no-rational} and~\ref{odd-no-rational} give
\(\mathrm{(H)},\mathrm{(I)}\). These results also prove the converses.

The families are disjoint by intrinsic invariants. In characteristic $5$,
inseparability isolates \(\mathrm{(A)}\); among separable maps, $G_f$ separates
cyclic from dihedral monodromy. In the cyclic case,
\(\mathrm{(A)},\mathrm{(B)},\mathrm{(C)}\) have respectively two rational, a
nonrational conjugate pair, and one rational totally ramified branch value.
In the dihedral case, total ramification separates \(\mathrm{(D)},\mathrm{(E)}\),
rationality of the branch locus separates \(\mathrm{(F)},\mathrm{(G)}\) from
\(\mathrm{(H)},\mathrm{(I)}\), and characteristic separates each remaining pair.
\end{proof}

Theorem~\ref{thm:intro-signed-descent} supplies the common descent input for
the nonclassical families. The elementary families are treated first.
For family $\mathrm{(I)}$, the Frobenius defect selects a distinguished
rational point of order two; adjoining it to the order-five kernel gives
a cyclic order-ten pair. Kubert's order-ten parameter $v$ yields the
generator-independent geometric invariant $m(v)$ of this pair. We use
$\Theta$ for a value of $m$ and later recover the same value directly
from the coefficients as $\Theta(\mathfrak u)$. Finite-field descent
then recovers the signed $k$-form and reduces the class count to a
quadratic-character count.

\begin{proof}[Proof of Theorem~\ref{thm:intro-equivalence-counts} (1).]
Family \(\mathrm{(A)}\) gives one class. By
Theorem~\ref{cyclic-classification}(2), all R\'edei functions
$R_{5,\delta}$ with $\delta\in\mathbb F_{q^2}\setminus\mathbb F_q$ are
$k$-M\"obius equivalent. Hence \(\mathrm{(B)}\) also gives one class. For
\(\mathrm{(C)}\)--\(\mathrm{(E)}\),
\begin{equation}
\begin{aligned}
 X^5-aX&\sim_k X^5-a'X
   &&\Longleftrightarrow& a'/a&\in(k^\times)^4,\\
 D_5(X,a)&\sim_k D_5(X,a')
   &&\Longleftrightarrow& a'/a&\in(k^\times)^2,\\
 X(X^2-a)^2&\sim_k X(X^2-a')^2
   &&\Longleftrightarrow& a'/a&\in(k^\times)^2.
\end{aligned}
\label{eq:elementary-equivalence}
\end{equation}
The unique totally ramified value and its preimage force both
coordinate changes to fix $\infty$, so
$f_{a'}(X)=wf_a(uX+v)+z$. In family $\mathrm{(C)}$, comparison of the
leading and linear terms gives $a'/a=u^{-4}$. In families
$\mathrm{(D)}$ and $\mathrm{(E)}$, the coefficients of $X^4$ and
$X^2$, respectively, give $v=0$, and then $a'/a=u^{-2}$. Scaling proves the converses; since all parameters in $\mathrm{(E)}$
are nonsquares, that family gives one class.
\end{proof}

\begin{proof}[Proof of Theorem~\ref{thm:intro-equivalence-counts} (2).]
The associated curve and kernel polynomial are
\[
 E:y^2=x^3+Ax^2+Bx+Ar,
 \qquad Q_N=X^2-r,
 \qquad a=4B+12r.
\]
By Lemma~\ref{lem:marked-Weierstrass}\textup{(1)}, an
origin-preserving geometric isomorphism between the associated marked
models has $\nu=\tau=0$. Write its scaling parameter as $u$. Since both
kernel polynomials have roots of sum zero, their affine correspondence
forces $\rho=0$. Consequently
$B'=u^{-4}B$ and $r'=u^{-4}r$; thus $B/r$, and therefore
$a/r=4B/r+12$, is intrinsic. Theorem~\ref{thm:intro-signed-descent}
proves necessity.

Conversely, if $a/r=a'/r'$, choose $c\in k^\times$ with
$c^{-2}r=r'$, possible because $r,r'$ are nonsquares. The rational
coordinate scaling $g(X)\mapsto c^{-1}g(cX)$ sends
\[
 (r,a,d)\longmapsto(c^{-2}r,c^{-2}a,c^{-5}d).
\]
The defining equation gives $d'=\pm c^{-5}d$; replacing $c$ by $-c$
realizes the required sign. Here $c=u^2$ when the coordinate scaling is
induced by the preceding Weierstrass isomorphism.
\end{proof}

\begin{proof}[Proof of Theorem~\ref{thm:intro-equivalence-counts} (3).]
The functions \(\mathcal U_t\) have one geometric branch point and
the functions \(\mathcal G_t\) have two, so the rows are disjoint.
If \(t,t'\) have trace one, choose \(\mu\in k\) with
\(\mu^2+\mu=t+t'\). Then
\(\mathcal U_t(X+\mu)+\mu=\mathcal U_{t'}(X)\).

For the ordinary row, Theorem~\ref{thm:intro-signed-descent} reduces
necessity to a signed marked isomorphism. The normalized datum is
\[
 E_t:y^2+\gamma xy=x^3+t^2(t+1),
 \qquad \gamma^2=t^{-1},
 \qquad Q_N=X^2+X+t.
\]
Write $u_0$ for the scaling parameter in the marked
Weierstrass isomorphism. Its affine action on kernel coordinates has
linear part $u_0^2$. Since the two roots of each kernel polynomial
differ by $1$, one has $u_0^2=1$, hence $u_0=1$ in characteristic $2$.
The full coefficient formulas
\[
 u_0a_1'=a_1+2\nu,
 \qquad
 u_0^3a_3'=a_3+\rho a_1+2\tau
\]
\cite[Table~3.1, p.~45]{Silverman2009ArithmeticEllipticCurves} give
$\gamma'=\gamma$ and $0=\rho\gamma$. Since
$\gamma^2=t^{-1}$, one has $\gamma\ne0$, so $\rho=0$. The affine map
on kernel coordinates is the identity; hence the kernel polynomials
coincide and $t=t'$. The converse is immediate.
\end{proof}

\begin{proof}[Proof of Theorem~\ref{thm:intro-equivalence-counts} (4).]
By Theorem~\ref{thm:intro-signed-descent}, necessity is tested on
signed origin-preserving isomorphisms carrying the kernel subgroup and
the Frobenius defect. Put
\[
 \eta=\frac{r^5(r+1)}{u^2},
 \qquad A=u+\eta,
 \qquad \beta=r^5(r+1).
\]
The corresponding datum is
\[
 E_{r,u}:y^2+xy=x^3+Ax^2+\beta,
 \qquad Q_N=X^2+rX+r^3,
 \qquad \varepsilon=(u,0).
\]
Lemma~\ref{lem:marked-Weierstrass}\textup{(2)} gives
$u_0=1$ and $\rho=0$. Hence the kernel polynomial is fixed and
$r=r'$.

Conversely, suppose $r=r'$, and define $A'$ from $u'$. Choose
$c\in\bar k$ with $c^2+c=A+A'$. The origin-preserving isomorphism
\[
 \alpha_c:E_{r,u}\longrightarrow E_{r,u'},
 \qquad (x,y)\longmapsto(x,y+cx)
\]
preserves the kernel subgroup. Since $c^q=c$ or $c+1$, one has
${}^{\sigma_q}\alpha_c=\alpha_c$ or $[-1]\alpha_c$; indeed
$\alpha_{c+1}=[-1]\alpha_c$. On either model the unique nonzero
geometric $2$-torsion point is rational, so
$|\mathfrak D(E)|=2$. The trace condition gives
$\varepsilon\notin2E(k)$, and Lemma~\ref{lem:defect-isomorphism}
therefore identifies its defect with the unique nonzero class. The map
$\alpha_c$ identifies the two rational $2$-torsion points and hence the
two defects. Thus Theorem~\ref{thm:intro-signed-descent} applies.
\end{proof}

The next lemma gives the Kummer calculation that distinguishes the
relevant rational $2$-torsion point.

\begin{lem}[Kummer--Weil evaluation]
\label{lem:kummer-weil}
Let \(k\) have odd order, let \(h\in k[X]\) be a monic separable
cubic, let \(E:Y^2=h(X)\), and let
\(T=(t,0)\in E[2](k)\setminus\{O_E\}\). Let
\(e_2:E[2]\times E[2]\to\boldsymbol{\mu}_2=\{\pm1\}\subset k\)
be the Weil pairing
\cite[\S III.8, especially Proposition~III.8.1]{Silverman2009ArithmeticEllipticCurves}.
It induces a perfect pairing
\[
 \mathfrak D(E)\times E[2](k)\longrightarrow\boldsymbol{\mu}_2,
 \qquad ([S],T')\longmapsto e_2(S,T').
\]
For \(P\in E(k)\), define
\[
 \delta_T(P):=
 \begin{cases}
 1,&P=O_E,\\
 h'(t),&P=T,\\
 x(P)-t,&P\notin\{O_E,T\}.
 \end{cases}
\]
\[
\hbox to \displaywidth{%
  \rlap{Then}\hfil
  $\displaystyle e_2\bigl(\mathfrak d_E(P),T\bigr)=\chi_2\bigl(\delta_T(P)\bigr)$,
  \hfil
}
\]
where \(\chi_2\) is the quadratic character of \(k\).
\end{lem}

\begin{proof}
Put \(V=E[2]\) and \(W=V^{\sigma_q}=E[2](k)\). For \(S\in V\) and
\(T'\in W\), Frobenius compatibility of the Weil pairing
\cite[Proposition~III.8.1]{Silverman2009ArithmeticEllipticCurves}
gives
\[
 e_2((\sigma_q-1)S,T')=e_2(S,T')^q e_2(S,T')^{-1}=1.
\]
Thus \((\sigma_q-1)V\subseteq W^\perp\). Both subspaces have dimension
\(2-\dim W\) over \(\mathbb F_2\), so they are equal; nondegeneracy of
\(e_2\) gives the perfect pairing.

For \(P=O_E\), both sides equal \(1\). Suppose
\(P\notin\{O_E,T\}\), choose \(2\omega=P\), and put \(f_T=X-t\). Then
\(\operatorname{div}(f_T)=2(T)-2(O_E)\), and the duplication formula gives
\[
 f_T\circ[2]=g_T^2,\qquad
 g_T:=\frac{(X-t)^2-h'(t)}{2Y}.
\]
For \(S\in E[2]\), the function \(g_T\circ\tau_S/g_T\) has trivial
divisor and is therefore constant. For these choices, the function-theoretic definition of the Weil
pairing \cite[\S III.8]{Silverman2009ArithmeticEllipticCurves} is exactly
the constant rational function $g_T(Q+S)/g_T(Q)$, where $Q$ denotes
the generic point (up to inversion, which is immaterial in
$\boldsymbol{\mu}_2$). Thus the data
\(\operatorname{div}(f_T)=2(T)-2(O_E)\) and \(f_T\circ[2]=g_T^2\)
identify $g_T\circ\tau_S/g_T$ with \(e_2(S,T)\). Taking
\(S=\omega^{\sigma_q}-\omega\) and evaluating at \(Q=\omega\) gives
\[
 e_2(S,T)=\frac{g_T(\omega+S)}{g_T(\omega)}
 =\frac{g_T(\omega)^{\sigma_q}}{g_T(\omega)}
 =\chi_2(f_T(P)),
\]
because \(g_T(\omega)^2=f_T(P)\in k^\times\). Hence
\(e_2(\mathfrak d_E(P),T)=\chi_2(x(P)-t)\).

It remains to take \(P=T\). Translating \(t\) to \(0\), write
\[
 E:Y^2=X(X^2+AX+B),\qquad T=(0,0),\qquad B=h'(0).
\]
For a generic point $Q$ on $E$, the chord through $Q$ and $T$ gives
$x(Q+T)=B/x(Q)$. If
\(2\omega=T\), then \(\omega+T=-\omega\), so \(x(\omega)^2=B\).
Since $h$ is separable, $B\ne0$. If $B$ is a square, then
$x(\omega)\in k$, so Frobenius sends $\omega$ to one of the two points
above this abscissa, namely $\omega$ or $-\omega=\omega+T$. Thus the
defect lies in $\langle T\rangle$, whose elements pair trivially with
$T$. If $B$ is a nonsquare, then
$x(\omega)^{q}=-x(\omega)\ne x(\omega)$, so the defect is neither
$O_E$ nor $T$. In the two-dimensional symplectic space $E[2]$, every
element outside $\langle T\rangle$ pairs nontrivially with $T$.
Therefore the pairing is $-1$. In both cases the value is
$\chi_2(B)$, as asserted.
\end{proof}

\begin{prop}[The order-ten Tate parameter]
\label{prop:tate-parameter}
Let \(K\) be an algebraically closed field of odd characteristic, let
\(E/K\) be an elliptic curve, and let \(C\subset E(K)\) be cyclic of
order \(10\). For every generator \(P\) of \(C\), there is a unique
finite \(v\in K\) such that \((E,P)\cong(E_v,P_v)\), where
\begin{equation}
\begin{gathered}
 D:=v^2-3v+1,
 \qquad d:=-\frac{v^2}{D},
 \qquad \gamma:=v(d-1),
 \qquad \beta:=\gamma d,\\
 E_v:Y^2+(1-\gamma)XY-\beta Y=X^3-\beta X^2,
 \qquad P_v:=(0,0).
\end{gathered}
\label{eq:tate-10}
\end{equation}
Conversely, if \(D\ne0\) and \(E_v\) is nonsingular, then \(P_v\)
has exact order \(10\). Put \(C_v:=\langle P_v\rangle\),
\[
 m(v):=\frac{v(v-1)}{v^2-3v+1},
 \qquad
 \jmath(v):=\frac{v-1}{2v-1}.
\]
On the nonsingular locus,
\[
 (E_v,C_v)\cong(E_w,C_w)
 \quad\Longleftrightarrow\quad
 w=v\ \text{or}\ w=\jmath(v)
 \quad\Longleftrightarrow\quad
 m(w)=m(v).
\]
In particular, $m(v)$ is a generator-independent invariant of the
geometric cyclic order-ten pair: two nonsingular Tate parameters lie in
the same fibre of $m$ exactly when the corresponding cyclic order-ten
pairs are geometrically isomorphic. Moreover,
\begin{equation}
 \operatorname{disc}(E_v)=
 \frac{v^{10}(v-1)^{10}(2v-1)^5(4v^2-2v-1)}
 {(v^2-3v+1)^{10}}.
\label{eq:tate-discriminant}
\end{equation}
\end{prop}

\begin{proof}
Move $P$ to $(0,0)$ and its tangent to $Y=0$. One obtains
$Y^2+a_1XY+a_3Y=X^3+a_2X^2$, with $a_2a_3\ne0$: smoothness gives
$a_3\ne0$. If $a_2=0$, the tangent $Y=0$ meets the cubic with
multiplicity three at $P=(0,0)$, so the chord--tangent law gives
$3P=O_E$, contrary to the exact order $10$. A general admissible
change has $x=u^2x'+r_0$ and $y=u^3y'+u^2s_0x'+t_0$; fixing $(0,0)$
gives $r_0=t_0=0$, while preserving the tangent $Y=0$ gives $s_0=0$. Thus
only scaling remains, and the unique scaling with $a_2=a_3$ gives
\[
 E_{b,c}:Y^2+(1-c)XY-bY=X^3-bX^2,
 \qquad P=(0,0).
\]
The addition law gives $2P=(b,bc)$ and $3P=(c,b-c)$. Exact order $10$
forces $bc(b-c)\ne0$: $b=0$ gives $2P=P$, $c=0$ gives
$3P=-P$, and $b=c$ gives $3P=-2P$. Put $d=b/c$, so $b=cd$ and
$d\ne1$. Adding $2P$ and $3P$ gives
\[
 x(5P)=\frac{cd(c-d+1)}{(d-1)^2},\qquad
 y(5P)=-\frac{c^2d(c-d^2+d)}{(d-1)^3}.
\]
Consequently,
\[
 2y(5P)+(1-c)x(5P)-cd
 =-\frac{cd}{(d-1)^3}
 \bigl(c^2(d+1)-3cd(d-1)+d(d-1)^2\bigr).
\]
Since $cd(d-1)\ne0$, the point $5P$ is two-torsion exactly when
\begin{equation}
 c^2(d+1)-3cd(d-1)+d(d-1)^2=0.
\label{eq:tate-order-ten-condition}
\end{equation}
Since $d\ne1$, put $v=c/(d-1)$. Then
$d(v^2-3v+1)+v^2=0$. If $D=v^2-3v+1$ vanished, this would give $v=0$, contrary to $D(0)=1$. Thus $D\ne0$ and
\[
 d=-v^2/D,\qquad c=v(d-1)=\gamma,
 \qquad b=cd=\beta.
\]
This recovers $v$ uniquely. No step in this derivation divides by
$3$ or $5$; the formulas agree with Kubert's order-ten row, with his
parameter $f=v$,
\cite[Table~3, p.~217]{Kubert1976UniversalBounds}.

Conversely, if $D\ne0$ and $E_v$ is nonsingular, computing the
generalized Weierstrass discriminant gives \eqref{eq:tate-discriminant}. Hence $v(v-1)(2v-1)\ne0$, and
$d-1=-(2v-1)(v-1)/D$ shows $bc(b-c)\ne0$. Equation
\eqref{eq:tate-order-ten-condition} holds identically. The preceding
coordinate formula is finite because $d\ne1$, so $5P_v\ne O_E$ and is
a nonzero point of order $2$. Since $\beta\ne0$ gives
$P_v\ne-P_v$, the point $P_v$ has exact order $10$. In characteristic $5$, $\langle2P_v\rangle$ therefore consists
of five distinct geometric points, so it is reduced and $E_v$ is
ordinary.

The addition law also gives
\[
 x(2P_v)=\beta,\qquad x(4P_v)=d(d-1),\qquad
 x(5P_v)=\frac{v^3(1-v)}D.
\]
For the cubic $H_v$ obtained by completing the square,
\begin{equation}
 \frac{(x(5P_v)-x(2P_v))(x(5P_v)-x(4P_v))}
 {H_v'(x(5P_v))}=m(v).
\label{eq:tate-invariant}
\end{equation}
Both numerator and denominator are multiplied by $u^4$ under
$x=u^2X+r$. Replacing $P$ by $-P$ changes no $x$-coordinate, while
$P\mapsto3P$ interchanges $x(2P)$ and $x(4P)$. Hence $m(v)$ depends
only on $(E_v,C_v)$. Moreover
\[
 m(v)-m(w)=-\frac{(v-w)(2vw-v-w+1)}
 {(v^2-3v+1)(w^2-3w+1)},
\]
so $m(v)=m(w)$ exactly when $w=v$ or $w=\jmath(v)$. The second case
is induced by $P_v\mapsto3P_v$: with
$s_v=-(v-1)(3v-1)/D$, the change
\begin{equation}
 x=(2v-1)^2X+\gamma,
 \qquad y=(2v-1)^3Y+(2v-1)^2s_vX+\beta-\gamma
\label{eq:tate-generator-change}
\end{equation}
identifies $(E_{\jmath(v)},P_{\jmath(v)})$ with $(E_v,3P_v)$.
Modulo sign, the generators of $C_v$ are represented by $P_v$ and
$3P_v$, completing the proof.
\end{proof}

We use effective Galois descent in the following standard form. If
$L/k$ is cyclic of degree $m$ and a projective $L$-scheme carries a
semilinear automorphism $F$ covering a generator of $\operatorname{Gal}(L/k)$
and satisfying $F^m=1$, then the resulting \emph{descent datum}
is effective \cite[Lemma~39.25.3, Tag 0CCJ]{StacksProject}. A stable
finite marked subscheme is finite, hence projective over the base by
\cite[Lemma~29.45.16, Tag 0B3I]{StacksProject}, so the same effectiveness
result applies to it. Compatible morphisms descend by the full
faithfulness of fpqc descent
\cite[Lemma~35.35.11, Tag 02W0]{StacksProject}.
Here ``semilinear'' refers to coefficient conjugation, not to the
absolute Frobenius morphism. In the applications below, the origin, group
law, and inclusion of the reduced marked subgroup are compatible with the
descent datum and hence descend by the same full-faithfulness statement.
For a pair $(E,C)$, write
$\operatorname{Aut}(E,C)$ for the origin-preserving automorphisms of $E$
stabilizing $C$ setwise.

\begin{lem}[Automorphisms of an order-ten pair]
\label{lem:tate-automorphisms}
Let $E$ be an elliptic curve over an algebraically closed field of odd
characteristic, and let $C\subset E$ be a reduced cyclic subgroup of
order $10$. Restriction induces injections
\[
 \operatorname{Aut}(E,C)\hookrightarrow\operatorname{Aut}(C),
 \qquad
 \operatorname{Aut}(E,C)\hookrightarrow\operatorname{Aut}(C[5]).
\]
This includes characteristic $5$.
\end{lem}

\begin{proof}
An automorphism acting trivially on $C[5]$ fixes its five distinct
geometric points. If $u$ is a nonidentity origin-preserving
automorphism, positivity of the degree on $\operatorname{End}(E)$ gives
\[
 \deg(u-[1])\leq(\sqrt{\deg u}+1)^2=4
\]
\cite[Corollary~III.6.3]{Silverman2009ArithmeticEllipticCurves}, so
$u-[1]$ cannot have five distinct zeros. Thus restriction to $C[5]$ is
injective, and restriction to $C$ is injective a fortiori.
\end{proof}

\begin{prop}[Finite-field forms of the order-ten pair]
\label{prop:tate-descent}
Let \(k=\mathbb F_q\) have odd characteristic, write
\(k_m:=\mathbb F_{q^m}\), and let \(\chi_2\) be the quadratic character of
\(k\), extended by \(\chi_2(0)=0\). For \(\Theta\in k\), put
\[
 \Delta_\Theta:=5\Theta^2-2\Theta+1.
\]
Suppose that the fibre \(m^{-1}(\Theta)\) contains a nonsingular Tate
parameter. All nonsingular parameters in this fibre represent a single
geometric isomorphism class of pairs; choose a representative and denote
it by \((E_\Theta,C_\Theta)\). This pair has a \(k\)-form for which
arithmetic Frobenius acts on
\(C_\Theta\) by a multiplier in \(\{3,-3\}\) if and only if
\(\chi_2(\Delta_\Theta)\in\{0,-1\}\). For \(p\ne5\), the singular
invariant values are \(0,1,1/5\); for \(p=5\), they are \(0,1\). On
\(C_\Theta[5]\), the multipliers \(3,-3\) restrict to \(-2,2\),
respectively.
\end{prop}

\begin{proof}
The fibre equation
\(m(v)=\Theta\) is
\begin{equation}
 (\Theta-1)v^2+(1-3\Theta)v+\Theta=0,
\label{eq:tate-fibre-quadratic}
\end{equation}
with discriminant \(\Delta_\Theta\). Let
\(J_v:E_{\jmath(v)}\to E_v\) be the isomorphism
\eqref{eq:tate-generator-change}; it sends
\(P_{\jmath(v)}\) to \(3P_v\). Lemma~\ref{lem:tate-automorphisms} applies to $(E_v,C_v)$. Consequently, whenever \(v^q=\jmath(v)\), coefficientwise conjugation
gives \(J_v^{\sigma_q}:E_v\to E_{\jmath(v)}\), and
\begin{equation}
 J_v\circ J_v^{\sigma_q}=[-1],
\label{eq:J-cocycle}
\end{equation}
because the composite sends \(P_v\) to \(9P_v=-P_v\), and
Lemma~\ref{lem:tate-automorphisms} identifies an automorphism of the
pair from its action on $C_v$.

If \(\Delta_\Theta\) is a nonsquare, the two roots of
\eqref{eq:tate-fibre-quadratic} lie in \(k_2\) and satisfy
\(v^q=\jmath(v)\). Over \(k_4\), let
\(\sigma_q:E_v\to E_{v^q}\) denote the semilinear base-conjugation
isomorphism induced by \(a\mapsto a^q\); it is not the absolute
Frobenius morphism. Set \(F=J_v\circ\sigma_q\). Then \(F\) is
\(\sigma_q\)-semilinear, \(F(P_v)=3P_v\), and
\eqref{eq:J-cocycle} gives \(F^2=[-1]\circ\sigma_q^2\), hence
\(F^4=1\). The preceding descent statement gives the required
\(k\)-form. With our arithmetic-Frobenius convention, the action on the
marked cyclic group is represented by $F$ and hence has multiplier $3$;
using the inverse convention replaces $3$ by $3^{-1}=-3$ modulo $10$ and
therefore does not change the set of allowed multipliers.

If \(\Delta_\Theta=0\), the double root \(v\) lies in \(k\) and
\(\jmath(v)=v\). Thus \(J_v\) is a \(k\)-automorphism with
\(J_v(P_v)=3P_v\) and \(J_v^2=[-1]\). Over \(k_4\), the semilinear
automorphism \(F=J_v\circ\sigma_q\) has order \(4\); with the same
convention, its descent again has arithmetic Frobenius multiplier \(3\) on
the cyclic group.

Under the standing nonsingularity hypothesis, \(\Theta\ne1\). If
\(\Delta_\Theta\) is a nonzero square, the two roots \(v\) and
\(\jmath(v)\) are distinct elements of \(k\). An automorphism of
\((E_v,C_v)\) acting by \(\pm3\) would, after composition with
\([-1]\) if necessary, identify \((E_v,P_v)\) with
\((E_v,3P_v)\), forcing \(v=\jmath(v)\). Hence
\(\operatorname{Aut}(E_v,C_v)=\{\pm1\}\). Here $v\in k$ and
$P_v=(0,0)\in E_v(k)$, so the standard pair is defined over $k$ and has
Frobenius multiplier $1$ on $C_v$. After a geometric identification with
it, the descent datum of any other $k$-form differs from the standard one
by a Frobenius cocycle with values in $\operatorname{Aut}(E_v,C_v)$. Thus
arithmetic Frobenius acts on $C_v$ by $\pm1$, never by $\pm3$.

It remains to remove singular fibres. By
Proposition~\ref{prop:tate-parameter}, every exact order-ten marked
pair has a finite parameter \(v\). Formula~\eqref{eq:tate-discriminant}
shows that the finite singular parameters are \(0,1,1/2\) and the
roots of \(4v^2-2v-1\); the roots of \(D\) are poles of the model.
Their finite \(m\)-values are \(0,1,1/5\) when \(p\ne5\), since
\[
 5v(v-1)-(v^2-3v+1)=4v^2-2v-1.
\]
For \(\Theta=1\), the second projective root of
\eqref{eq:tate-fibre-quadratic} is \(v=\infty\), but no exact
order-ten pair has an infinite Tate parameter; the finite root
\(v=1/2\) is singular. In characteristic \(5\), the roots of
\(4v^2-2v-1\) are roots of \(D\), so only \(0\) and \(1\) remain.
\end{proof}

\medskip
\noindent\textbf{The coefficient realization of the order-ten invariant.}
Let \(\mathfrak u:=(a,b,c,r,s)\) be odd-admissible, and put
\[
\begin{aligned}
 E_{\mathfrak u}&:y^2+cy=x^3+ax^2+bx,\\
 h_{\mathfrak u}(X)&:=X^3+aX^2+bX+\frac{c^2}{4},\\
 Q_{\mathfrak u}(X)&:=X^2+rX+s.
\end{aligned}
\]
Let $\varepsilon_{\mathfrak u}:=(0,0)$, and let
$N_{\mathfrak u}\subset E_{\mathfrak u}(\bar k)$ be the cyclic
order-five subgroup with kernel polynomial $Q_{\mathfrak u}$.

\begin{prop}[The distinguished two-torsion point]
\label{prop:I-distinguished}
The polynomial $h_{\mathfrak u}$ has one or three roots in $k$. If it
has one, denote that root by $t_{\mathfrak u}$. If it has three, there
is a unique root $t_{\mathfrak u}$ for which
\begin{equation}
\label{eq:I-distinguished}
 \delta_{\mathfrak u}(t):=
 \begin{cases}
 -t,&t\ne0,\\
 b,&t=0
 \end{cases}
 \quad\text{belongs to }(k^\times)^2.
\end{equation}
The point
$T_{\mathfrak u}:=(t_{\mathfrak u},-c/2)$ is intrinsic and functorial
under signed isomorphisms preserving the kernel and the Frobenius
defect. Moreover
\begin{equation}
\label{eq:I-theta-invariant}
 \Theta(\mathfrak u)
 :=\frac{Q_{\mathfrak u}(t_{\mathfrak u})}
 {h'_{\mathfrak u}(t_{\mathfrak u})}
 =\frac{t_{\mathfrak u}^2+rt_{\mathfrak u}+s}
 {3t_{\mathfrak u}^2+2at_{\mathfrak u}+b}
\end{equation}
is well defined.
\end{prop}

\begin{proof}
Put $E=E_{\mathfrak u}$, $\varepsilon=\varepsilon_{\mathfrak u}$, and
$N=N_{\mathfrak u}$. Odd-admissibility gives
$\varepsilon\notin2E(k)$. After completing the square, the nonzero
rational points of order $2$ are
\[
 T_t=(t,-c/2),\qquad h_{\mathfrak u}(t)=0.
\]
Lemma~\ref{lem:defect-isomorphism} gives a nonzero defect and hence
$E[2](k)\ne\{O_E\}$, so $h_{\mathfrak u}$ has one or three roots in
$k$. By Lemma~\ref{lem:kummer-weil},
\[
 e_2(\mathfrak d_E(\varepsilon),T_t)
 =\chi_2(\delta_{\mathfrak u}(t)).
\]
If $|E[2](k)|=2$, its nonzero point is unique. If $|E[2](k)|=4$, the
kernel of the nontrivial character
$T\mapsto e_2(\mathfrak d_E(\varepsilon),T)$ contains exactly one
nonzero point. This proves existence and uniqueness. Functoriality follows because a
signed isomorphism preserving the defect is Galois equivariant on $E[2]$
and, by naturality of the Weil pairing,
$e_2(\alpha_*\mathfrak d_E(\varepsilon),\alpha(T))=
e_2(\mathfrak d_E(\varepsilon),T)$; hence it carries the distinguished
point to the distinguished point. The values in
\eqref{eq:I-distinguished} are nonzero. For $t\ne0$ this follows from the definition, while $t=0$ gives $c=0$ and nonsingularity gives $b\ne0$.
Finally $h'_{\mathfrak u}(t_{\mathfrak u})\ne0$ by nonsingularity and
$Q_{\mathfrak u}(t_{\mathfrak u})\ne0$ by irreducibility, so
$\Theta(\mathfrak u)$ is defined.
\end{proof}

\begin{prop}[The geometric order-ten invariant]
\label{prop:I-geometric-invariant}
Put
\[
 C_{\mathfrak u}:=N_{\mathfrak u}\oplus\langle T_{\mathfrak u}\rangle.
\]
Then $C_{\mathfrak u}$ is cyclic of order $10$. If $v$ is the Tate
parameter associated by Proposition~\ref{prop:tate-parameter} to any
generator of $C_{\mathfrak u}$, then
\[
 \Theta(\mathfrak u)=m(v).
\]
Consequently, for two odd-admissible quintuples
$\mathfrak u,\mathfrak u'$ one has
\[
 (E_{\mathfrak u},C_{\mathfrak u})\cong_{\bar k}
 (E_{\mathfrak u'},C_{\mathfrak u'})
 \quad\Longleftrightarrow\quad
 \Theta(\mathfrak u)=\Theta(\mathfrak u').
\]
\end{prop}

\begin{proof}
The group $C_{\mathfrak u}$ recovers its unique subgroups of orders $5$
and $2$. For a generator $P$ of $C_{\mathfrak u}$ one has
$N_{\mathfrak u}=\langle2P\rangle$ and $T_{\mathfrak u}=5P$. Thus the
roots of $Q_{\mathfrak u}$ are $x(2P)$ and $x(4P)$, while
$t_{\mathfrak u}=x(5P)$. On the completed-square model,
$Q_{\mathfrak u}(t_{\mathfrak u})$ is therefore
$(x(5P)-x(2P))(x(5P)-x(4P))$, and
$h'_{\mathfrak u}(t_{\mathfrak u})$ is the derivative of the cubic at
$x(5P)$. These are the numerator and denominator in
\eqref{eq:tate-invariant}; under an affine change of the $x$-coordinate,
both are multiplied by the same fourth power. Hence
$\Theta(\mathfrak u)=m(v)$ for the Tate parameter attached to any
generator of $C_{\mathfrak u}$. Proposition~\ref{prop:tate-parameter}
then gives the equivalence.
\end{proof}

\begin{prop}[Arithmetic realization and defect recovery for family \(\mathrm{(I)}\)]
\label{prop:I-arithmetic-realization}
Let $(E,C)$ be a $k$-form of a nonsingular geometric order-ten pair on
which Frobenius acts by $\pm3$. Put $N:=C[5]$ and let $T$ be the
nonzero point of $C[2]$. There is a unique nonzero class
$\delta_{E,C}\in\mathfrak D(E)$: it is the unique nonzero class if
$|E[2](k)|=2$, and, if $|E[2](k)|=4$, it is characterized by
$e_2(\delta_{E,C},T)=1$. For every $\varepsilon\in E(k)$ with
$\mathfrak d_E(\varepsilon)=\delta_{E,C}$, the datum
$(E,N,\varepsilon)$ yields a member $\mathcal I_{\mathfrak u}$ of
family $\mathrm{(I)}$ whose distinguished point is $T$ and whose
$\Theta$-invariant is the geometric order-ten invariant of $(E,C)$. Its $k$-M\"obius class is independent of the representative of
$\varepsilon$ and of the subsequent coordinate choices. Conversely, the defect attached to
an odd-admissible quintuple is $\delta_{E,C}$; hence every signed
isomorphism of the corresponding order-ten pairs preserves it.
\end{prop}

\begin{proof}
The multiplier $\pm3$ fixes $T$ and restricts to $\mp2$ on $N$. Hence
$T\in E[2](k)\setminus\{O_E\}$. If $|E[2](k)|=2$, then
Lemma~\ref{lem:defect-isomorphism} shows that $\mathfrak D(E)$ has a
unique nonzero element. If $|E[2](k)|=4$, the perfect pairing in
Lemma~\ref{lem:kummer-weil} shows that the annihilator of $T$ in
$\mathfrak D(E)$ has order $2$, and therefore contains a unique
nonzero element. This is precisely the class $\delta_{E,C}$ in the
statement.

By Lemma~\ref{lem:defect-isomorphism}, choose
$\varepsilon\in E(k)$ with
$\mathfrak d_E(\varepsilon)=\delta_{E,C}$. Then
$\varepsilon\notin2E(k)$. If $\varepsilon\in N$, rationality gives
$\sigma_q(\varepsilon)=\varepsilon$, whereas Frobenius acts on $N$ by
$[\lambda]$ with $\lambda\in\{2,-2\}\pmod5$. Thus
$(\lambda-1)\varepsilon=O_E$; since $\lambda-1$ is a unit modulo $5$,
this forces $\varepsilon=O_E$, contrary to the nonzero defect. Hence
$\varepsilon\notin N$.

Theorem~\ref{thm:intro-dihedral} now gives an exceptional $D_5$-map
without a rational branch point, and Theorem~\ref{odd-no-rational}
puts it in family $\mathrm{(I)}$. When $|E[2](k)|=2$, the point $T$
is the unique nonzero rational two-torsion point. When $|E[2](k)|=4$,
the defining relation $e_2(\delta_{E,C},T)=1$ and
Lemma~\ref{lem:kummer-weil} identify $T$ as the distinguished point of
that family. Formula~\eqref{eq:tate-invariant} then gives its $\Theta$-invariant.
Changing the representative of $\varepsilon$ does not change its class
in $E(k)/2E(k)$ and hence does not change the lower $k$-M\"obius class
by Theorem~\ref{thm:intro-dihedral}\textup{(2)}; changing quotient
coordinates only left- and right-composes by elements of
$\operatorname{PGL}_2(k)$.

Conversely, an odd-admissible quintuple has nonzero defect, and its
distinguished point is characterized exactly as above. Thus its defect
is $\delta_{E,C}$. A signed isomorphism of the corresponding
order-ten pairs preserves $T$ and the Weil pairing, and therefore
preserves this defect.
\end{proof}

\begin{proof}[Proof of Theorem~\ref{thm:intro-equivalence-counts} (5).]
A \(k\)-M\"obius equivalence gives, by Theorem~\ref{thm:intro-signed-descent}, a
signed isomorphism preserving the kernel and defect. It therefore
preserves $T_{\mathfrak u}$ and $C_{\mathfrak u}$, so
Proposition~\ref{prop:I-geometric-invariant} gives equality of the
$\Theta$-invariants.

Conversely, equality of the invariants gives a geometric isomorphism
$\alpha:(E,C)\to(E',C')$. Since $C,C'$ are Frobenius-stable,
$\beta=({}^{\sigma_q}\alpha)\alpha^{-1}$ stabilizes $C'$. If
$P\in N$ and $Q=\alpha(P)\in N'$, then
\[
 \beta(Q)=\sigma_q\alpha(\sigma_q^{-1}P)
 = [\lambda'\lambda^{-1}]Q.
\]
Thus $\beta|_{N'}=\lambda'\lambda^{-1}\in\{\pm1\}$.
Lemma~\ref{lem:tate-automorphisms} makes restriction to $C'[5]=N'$
injective, so $\beta=1$ or $[-1]$ and
${}^{\sigma_q}\alpha=\pm\alpha$.
Proposition~\ref{prop:I-arithmetic-realization} gives defect preservation,
and Theorem~\ref{thm:intro-signed-descent} gives
$\mathcal I_{\mathfrak u}\sim_k\mathcal I_{\mathfrak u'}$.
\end{proof}

\begin{proof}[Proof of Theorem~\ref{thm:intro-equivalence-counts} (6).]
Families $\mathrm{(A)}$ and $\mathrm{(B)}$ each give one class. Since
$k^\times$ is cyclic \cite[Theorem~2.8]{LidlNiederreiter1997FiniteFields},
its three nontrivial cosets modulo $(k^\times)^4$ give
$N_{\mathrm C}=3$ in characteristic $5$, while
\eqref{eq:elementary-equivalence} gives
$N_{\mathrm D}=\gcd(2,q-1)$ and $N_{\mathrm E}=1$.

For $\mathrm{(F)}$, put $x=a/r$ and $h(X)=X^2-16X+128$; then
$d^2=r^5h(x)$. Since $r^5$ is nonsquare, this equation has a solution
$d\in k$ exactly when $\chi_2(h(x))\in\{0,-1\}$. Part~\textup{(2)}
shows that the resulting $k$-M\"obius class depends only on $x$. In general, if $N_0=|\{x:h(x)=0\}|$ and
$S=\sum_x\chi_2(h(x))$, then
$|\{x:\chi_2(h(x))\in\{0,-1\}\}|=(q+N_0-S)/2$. Here
\cite[Theorem~5.48]{LidlNiederreiter1997FiniteFields} gives
$S=-1$ and $N_0=1+\chi_2(-1)$, hence
$(q+2+\chi_2(-1))/2$ values before the totally ramified boundary is
removed. The value $x=52$, when it occurs, is still exceptional but
is already represented by family \textup{(D)} or \textup{(E)}; see
Remark~\ref{rem:family-F-boundary}. Since
$h(52)=2000=20^2\cdot5$, it is counted exactly when
$\chi_2(5)\ne1$. Hence
$N_{\mathrm F}=(q+2+\chi_2(-1))/2-\epsilon_5$.

If $p=2$, the trace-one fibre has $q/2$ elements. Parts~\textup{(3)} and~\textup{(4)} give
$N_{\mathrm G}=q/2+1-\epsilon_2$ and
$N_{\mathrm H}=q/2-\epsilon_2$; every admissible $r$ in
$\mathrm{(H)}$ occurs because squaring permutes $k^\times$.

For $\mathrm{(I)}$, Frobenius acts by $\pm2$ on $C_{\mathfrak u}[5]$ and trivially on $C_{\mathfrak u}[2]$, hence by $\pm3$ on $C_{\mathfrak u}$ by the Chinese remainder theorem. Propositions~\ref{prop:tate-parameter} and~\ref{prop:tate-descent} now show that a nonsingular invariant value $\Theta\in k$ occurs
exactly when $\chi_2(\Delta_\Theta)\in\{0,-1\}$. For every such value,
Proposition~\ref{prop:I-arithmetic-realization} produces an
odd-admissible quintuple, while part~\textup{(5)} shows that all
realizations with that value give one and the same class.

If $p\ne5$, the discriminant of the quadratic
$\Delta_\Theta=5\Theta^2-2\Theta+1$ is $-16$. Hence
\cite[Theorem~5.48]{LidlNiederreiter1997FiniteFields} gives
$\sum_{\Theta\in k}\chi_2(\Delta_\Theta)=-\chi_2(5)$, while
$|\{\Theta:\Delta_\Theta=0\}|=1+\chi_2(-1)$. Moreover
$\Delta_0=1$, $\Delta_1=4$, and $\Delta_{1/5}=4/5$; hence $0,1$
are never counted, whereas $1/5$ is counted exactly when
$\chi_2(5)=-1$. Removing the singular values gives
$N_{\mathrm I}=(q+\chi_2(-1)+2\chi_2(5))/2$. If $p=5$, then
$\Delta_\Theta=1-2\Theta$ vanishes once and is nonsquare $(q-1)/2$ times,
while $0,1$ are not counted; hence $N_{\mathrm I}=(q+1)/2$.
\end{proof}

\end{document}